\documentclass[fontsize=11pt,a4paper,DIV=17]{scrartcl}

\usepackage[utf8]{inputenc}
\usepackage{amsthm,amssymb,amsmath,amsfonts}
\usepackage{graphicx}
\usepackage[labelformat=simple]{subcaption}
\usepackage{lineno}
\usepackage{thmtools}
\usepackage{thm-restate}

\usepackage[usenames,dvipsnames]{xcolor}

\usepackage[unicode,colorlinks=true,urlcolor=blue,citecolor=blue,linkcolor=black]{hyperref}
\usepackage[capitalise,noabbrev,nameinlink]{cleveref}
\hypersetup{hidelinks}

\usepackage{enumerate}

\usepackage{microtype}
\usepackage{layout}
\usepackage{wrapfig}
\usepackage{bibitemsep}

\usepackage[font=small,labelfont=bf]{caption}

\usepackage{mathrsfs}
\usepackage[scr=rsfs,cal=boondox]{mathalfa}

\usepackage{mathtools}

\newtheorem{theorem}{Theorem}
\newtheorem{lemma}[theorem]{Lemma}

\newtheorem{corollary}[theorem]{Corollary}
\newtheorem*{claim*}{Claim}
\newtheorem*{remark*}{Remark}
\newtheorem{question}{Question}
\newtheorem{conjecture}{Conjecture}

\newtheorem{claim}{Claim}

\graphicspath{{./figs/}}

\def\AA{{\cal A}}
\def\BB{{\cal B}}
\def\CC{{\cal C}}
\def\DD{{\cal D}}

\def\SS{{\cal S}}
\def\arrminus{{-}}

\def\AAsixB{{\cal N}_6^2}

\newcommand{\interior}[1]{\mathrm{int}(#1)}
\newcommand{\exterior}[1]{\mathrm{ext}(#1)}

\definecolor{myblue}{rgb}{0.19, 0.55, 0.91}
\newcommand{\defi}[1]{\textcolor{myblue}{\emph{#1}}}

\def\inst#1{$^{#1}$}

\usepackage{soul}

\title{Flip Graph Connectivity 
for~Arrangements of~Pseudolines~and~Pseudocircles\thanks{%
This work was initiated at a workshop of the collaborative DACH project \emph{Arrangements and  Drawings} in Lutherstadt Wittenberg in April 2022 and parts appeared in the Bachelor's thesis of the first author~\cite{AlvesRadtke2023}. 
We thank all the participants for the  inspiring atmosphere.
S.F.\ was supported by DFG Grant FE~340/13-1. 
J.O.\ was supported by ERC StG~757609.
S.R.\ was supported by DFG-Research Training Group 'Facets of Complexity' (DFG-GRK~2434).
M.S.\ was supported by DFG Grant SCHE~2214/1-1.
}}

\author{
Yan Alves Radtke\inst{1}
\and 
Stefan Felsner\inst{1}
\and 
Johannes Obenaus\inst{2}
\and 
Sandro Roch\inst{1} 
\and
Manfred Scheucher\inst{1}
\and
Birgit Vogtenhuber\inst{3}
}
\date{} 

\begin{document}
\thispagestyle{empty}

\maketitle

\begin{center}
\vskip-12mm
{\footnotesize
\inst{1} 
Institut f\"ur Mathematik, 
Technische Universit\"at Berlin, Germany\\
\texttt{\{felsner,roch,scheucher\}@math.tu-berlin.de}
\medskip

\inst{2} 
Department of Computer Science, 
Freie Universit\"at Berlin, Germany,\\
\texttt{\{johannes.obenaus\}@fu-berlin.de}
\medskip

\inst{3} 
Institute of Software Technology, Graz University of Technology, Graz, Austria\\
\texttt{\{bvogt\}@ist.tugraz.at}
}
\end{center}

\vfill

\begin{abstract}
\centerline{\textbf{Abstract}}
\medskip

\noindent
Flip graphs of combinatorial and geometric objects are at the heart of many
deep structural insights and connections between different branches of
discrete mathematics and computer science. They also provide
a natural framework for the study of reconfiguration problems.
We study flip graphs of
arrangements of pseudolines and of arrangements of pseudocircles, which are
combinatorial generalizations of lines and circles, respectively.
In both cases we consider triangle flips as local transformation and 
prove conjectures regarding their connectivity.

In the case of $n$ pseudolines we show that the connectivity of the flip graph equals its minimum degree, which is exactly $n-2$. 
For the proof we introduce the class of
\emph{shellable} line 
arrangements, which serve as reference objects for the
construction of disjoint paths. In fact, shellable arrangements are elements of
a flip graph of line arrangements which are vertices of a polytope (Felsner and
Ziegler; DM 241 (2001), 301--312).
This polytope forms a cluster of good connectivity in the flip graph of pseudolines.
In the case of pseudocircles we show that triangle flips induce a connected flip graph on \emph{intersecting} arrangements and also on \emph{cylindrical}
intersecting arrangements. 
The result for cylindrical arrangements is used in the proof for intersecting arrangements. 
We also show that in both settings the diameter of the flip graph is in~$\Theta(n^3)$.
Our constructions make essential use of variants of the sweeping lemma for pseudocircle arrangements
(Snoeyink and Hershberger; Proc.\ SoCG 1989: 354--363). 
We finally study cylindrical arrangements in their own right and provide new combinatorial
characterizations of this class.
\end{abstract}

\vfill

\newpage

%%%%%%%%%%%%%%%%%%%%%%%%%%%%%%%%%%%%%%%%%%%%%%%%%%
%%%%%%%%%%%%%%%%%%%%%%%%%%%%%%%%%%%%%%%%%%%%%%%%%%
\section{Introduction}
\label{sec:introduction}
%%%%%%%%%%%%%%%%%%%%%%%%%%%%%%%%%%%%%%%%%%%%%%%%%%
%%%%%%%%%%%%%%%%%%%%%%%%%%%%%%%%%%%%%%%%%%%%%%%%%%

Reconfiguration is a widely studied topic in discrete mathematics and theoretical computer science~\cite{Nishimura2018survey}.
In many cases, reconfiguration problems can be stated in terms of a \emph{flip graph}. 
For a class of objects, the flip graph has a vertex for each object and adjacencies are determined by a local flip operation, which transforms one object into another.
Typically, the first question is whether a flip graph is connected. 
In the affirmative case, more refined questions
regarding radius, diameter, degree of connectivity, shortest paths, mixing properties, or Hamiltonicity can be of interest. 
Hamiltonicity of flip graphs is related to combinatorial generation and Gray codes, cf.\ Knuth's~TAOCP vol.4A~\cite{Knuth4A} 
and the survey on Gray codes by M\"utze~\cite{Muetze2022survey}.
In nice cases flip graphs turn out to be the graphs of polytopes, 
cf.~recent work on graph associahedra~\cite{MannevillePilaud2015} 
and quotientopes~\cite{PilaudSantos2019}. 
This yields enriching connections between combinatorics and geometry.

A prototypical example is the flip graph of triangulations of a convex polygon.
The vertex set of this graph are all triangulations of the polygon, and two triangulations are adjacent if one can be obtained from the other by exchanging the common edge of two adjacent triangles by the other diagonal of the convex quadrilateral formed by them. 
Similar flip graphs have also been investigated in the context of Delaunay triangulations~\cite{EdelBook01} and for triangulations of 
point sets~\cite{LRSTriangulationsBook2010}.

Two arrangements of curves are related by a \defi{triangle 
flip} if one is obtained from the other by changing the 
orientation of a triangular cell, i.e., moving one of the three
curves of the triangle across the intersection of the two others.
On a family of arrangements we obtain the \defi{triangle flip graph}
by taking the arrangements as vertices and triangle flips as edges.
Triangle flips\footnote{They are sometimes  
also referred to as \emph{mutations}~\cite{BjoenerLVWSZ1993}, \emph{homotopy moves}~\cite{celmsstt-tcslm-18}, or \emph{Reidemeister moves of type~III}~\cite{l-pubrm-15}, where the latter specifies an above/below relationship between the curves/strands.}
on arrangements of curves and the resulting flip graphs
have been studied in various contexts, including the transformation of drawings of graphs~\cite{AichholzerCH0KM23,gioan_proof,schaefer21}, 
curves on surfaces with large genus~\cite{celmsstt-tcslm-18}, and 
knot theory~\cite{l-pubrm-15,t-rmck-83}. 
The work at hand studies triangle flip graphs on arrangements of pseudolines and pseudocircles in the Euclidean plane.

An \defi{arrangement of pseudolines} is a finite collection of simple bi-infinite curves (\defi{pseudolines}) in the plane such that every pair of curves intersects in a unique point 
where the two curves cross. 
In a \defi{marked} arrangement, one of the unbounded cells is distinguished (see also \Cref{sec:prelim}).
An \defi{arrangement of pseudocircles} is a finite collection of simple closed curves (\defi{pseudocircles}) in the plane such that every pair of curves is either disjoint, touches in a single point, or intersects in exactly two crossing points.
An arrangement of pseudocircles	is \defi{intersecting} if every pair of curves in the arrangement has two crossing points.
An arrangement of pseudolines or pseudocircles 
is \defi{simple} if no three of its curves
share a common point. In this work we only consider
simple arrangements.

Arrangements of pseudolines have been introduced by Levi~\cite{Levi1926} as a generalization of arrangements of straight lines. Since then they have been intensively studied and used as they occur in many problems of computational geometry; see the handbook chapters~\cite{AGARWAL200049, FelsnerGoodman2018}.

Ringel~\cite{Ringel1956} studied arrangements of pseudolines using triangle flips. His \emph{homotopy theorem} (name coined in~\cite{BjoenerLVWSZ1993}) shows that the triangle flip graph of arrangements of pseudolines is connected.

While Levi investigated arrangements in the projective plane, Ringel was considering the Euclidean plane as ambient space. 
The number of triangles in arrangements, i.e., degrees in the  triangle flip graph have been studied since Levi's seminal work.
The number of triangles in Euclidean arrangements of $n$ pseudolines ranges from $n-2$ to $\frac{1}{3}n(n-2)$, where examples for both extremes exist for infinitely many values of $n$, for details see~\cite{FeKri99} (lower bound) and \cite{Blanc2011} (upper bound).

Manin and Schechtmann~\cite{MaSch89} introduced higher Bruhat orders, see also 
Ziegler \cite{Zi93}. Elements of higher Bruhat orders are signotopes: 
an \emph{$r$-signotope} is a mapping $\sigma: \binom{[n]}{r} \to \{+,-\}$
fulfilling a monotonicity condition on $(r+1)$-sets.
Felsner and Weil \cite{FelsnerWeil2001} showed that 3-signotopes encode arrangements of pseudolines. A flip between signotopes corresponds to changing the sign of a single $r$-set.
The resulting flip graph of $r$-signotopes is connected for any rank $r$. At the end of their paper Felsner and Weil ventured the following conjecture (the $r=1$ and $r=2$ cases of the statement are known to be true, they concern degree and connectivity of hypercubes and~permutahedra).

\begin{conjecture}[{\cite{FelsnerWeil2001}}]
\label{conjecture:signotopes}
For every $r\ge 3$, the flip graph $\mathbf{G}^r_n$ of $r$-signotopes on $n$ elements 
has minimum degree $n-r+1$ and is $(n-r+1)$-connected.
\end{conjecture}

The flip graph $\mathbf{G}^r_n$ of $r$-signotopes on $n$ elements has vertices of degree $n-r+1$, e.g.~the all-plus signotope.
Hence, the conjecture has two parts, show that the minimum degree is $n-r+1$ and show that the connectivity equals the minimum degree.  
The mapping from pseudoline arrangements with a marked north-cell
to 3-signotopes is as follows: 
The north-cell defines the upper side of a pseudoline, i.e., an above/below relation of 
pseudolines and points not on the line.
For a triple $\{i,j,k\}$ with $i < j < k$, let $\sigma(ijk) = +$ if the crossing of the lines $\ell_i$ and $\ell_k$ is above $\ell_j$ and $\sigma(ijk) = -$ if it is below. Changing a single sign thus corresponds to a triangle flip in the arrangement. 
The flip graph $\mathbf{G}^3_n$
of $3$-signotopes is the triangle flip graph of arrangements of $n$ pseudolines. 
The degree of a 3-signotope in the flip graph equals the number of triangles in the corresponding arrangement. As already noted, this number is at least $n-2$. We prove the connectivity
part of ~\Cref{conjecture:signotopes} for $r=3$ and
sufficiently large~$n$.

\begin{theorem}
\label{thm:connectivity}
For $n\leq 7$ and $n\geq 24$, the triangle flip graph $\mathbf{F}_n$ on marked arrangements of $n$ pseudolines is $(n-2)$-connected. 
\end{theorem}

Since pseudoline arrangements are closely related to several other combinatorial structures such as great-pseudocircle arrangements (oriented matroids of rank 3) and reduced decompositions  of permutations, we expect \Cref{thm:connectivity} to be of interest for multiple fields (see also \Cref{sec:conclusion}).

For the proof of Theorem~\ref{thm:connectivity} we will make use of a structural result by Felsner and Ziegler~\cite{FeZi01}.
They construct a high-dimensional zonotope such that its skeleton is the triangle flip graph of all arrangements of lines which can be realized with the prescribed slopes. 
Moreover, we will define \emph{shellable} arrangements (see \Cref{sec:prelim}) and
show that they can be realized with any set of prescribed slopes. 

An arrangement on $n$ lines can be seen as a real-valued $2 \times n$ matrix, where the $i$-th column encodes the slope and the $y$-intercept of the $i$-th line. 
For a given realizable pseudoline arrangement, the \defi{realization space} is the set of all matrices corresponding to realizing line arrangements. 
Due to  Mn{\"e}v's Universality Theorem \cite{Mnev1988}, realization spaces can behave topologically as bad as arbitrary semialgebraic sets.
In particular, there exist arrangements with disconnected realization space for $n \ge 13$; see \cite{Suvorov1988}.
In other words, a pair of line arrangements, both realizing the same combinatorial type, might not be  transformable into each other via continuous motion of lines without changing the combinatorial type.
Similarly if two line arrangements only differ by a single triangle orientation a continuous transformation from one to the other may have to pass
at a line arrangement of a third combinatorial type.
The following result is in contrast to these considerations. 

\begin{restatable}{proposition}{propConnReal}\label{prop:conn-real}
    The triangle flip graph $\mathbf{FL}_n$ on arrangements of $n$ lines in the plane is 
    $(n-2)$-connected for all $n \ge 3$.  Moreover, for every pair of line arrangements $\mathcal{L}$ and $\mathcal{L}'$, 
    and any $n-3$ pseudoline arrangements $\AA_1,\dots,\AA_{n-3}$ combinatorially different from $\mathcal{L}$ and $\mathcal{L}'$, 
    there exists a continuous motion of lines transforming $\mathcal{L}$ into $\mathcal{L}'$ such that every intermediate line arrangement avoids the combinatorial types of $\AA_1,\dots,\AA_{n-3}$. The statement also holds if  $\mathcal{L}$ and $\mathcal{L}'$ have the same combinatorial type.
\end{restatable}

The study of arrangements of pseudocircles as a generalization of arrangements of circles 
was initiated by Grünbaum \cite{Gruenbaum1972} in the 1970's. 
If disjoint pseudocircles are allowed, it is natural 
to consider three different flip operations:
In addition to the triangle flip, there is the \emph{digon-create} flip to allow two initially disjoint pseudocircles to start intersecting in a digon; and the reverse operation, called \emph{digon-collapse} flip (see \Cref{fig:flips} and \Cref{sec:prelim}).
With these three flip operations, the flip-connectivity of arrangements of 
circles is evident: one can shrink the circles until they have pairwise disjoint interiors.

The sweeping lemma of Snoeyink and Hershberger~\cite{SnoeyinkHershberger1991} implies
that any pseudocircle in an arrangement can be shrunk to a point in a continuous process which maintains the intersection requirements \mbox{($|C\cap C'| \leq 2$)} of pseudocircles. In a discretized form, this yields a sequence of digon and triangle flips. Hence, any two 
arrangements of $n$ pseudocircles can be transformed into each other via a sequence of these flips. 

However, 
the resulting flip sequence in general uses many digon flips and goes via the canonical arrangement of pairwise disjoint pseudocircles.
This naturally raises the question of whether it is possible to transform two intersecting arrangements
of pseudocircles into each other without any two pseudocircles becoming disjoint, or, in other words, by using only triangle flips.

\begin{question}
\label{question:pwi}
	Is the triangle flip graph of 
	intersecting arrangements of $n$ pseudocircles connected?
\end{question}

For intersecting arrangements of circles the answer 
is yes. We leave the details as an exercise to the reader. Another class where the answer is yes are
\emph{great-pseudocircle arrangements}, which are a generalization of great-circle arrangements on the sphere. These are intersecting arrangements with the property that for any three pseudocircles $C_i$, $C_j$, $C_k$ the two intersection points of $C_i$ and~$C_j$ are on different sides of $C_k$.
Great-pseudocircle arrangements are in bijection to pseudoline arrangements and hence Ringel's theorem carries over\footnote{For great-pseudocircle arrangements, a flip must consist of two simultaneous triangle flips applied to a pair of `antipodal' triangles, otherwise intermediate arrangements fail to have the great-pseudocircle property.}.
Felsner and Scheucher \cite{FelsnerScheucher2019} showed flip-connectivity for intersecting digon-free arrangements of proper circles and conjectured connectivity of the pseudocircular analogue (cf. {\cite[Conjecture~8.6]{FelsnerScheucher2019}}). 
Our second main result is a positive answer to \Cref{question:pwi}.

\begin{restatable}{theorem}{thmMainIpa}\label{thm:main_ipa}
The triangle flip graph of intersecting arrangements of $n$ pseudocircles is connected. 
\end{restatable}

A crucial ingredient
for our proof of \cref{thm:main_ipa} are
\defi{cylindrical arrangements}. These are arrangements of pseudocircles in the plane such that there is a point which belongs
to the bounded interior of each
pseudocircle, such a point is called a \defi{center}.
One part of the proof is to show that 
every intersecting cylindrical arrangement can be flipped into a canonical arrangement by
only using triangle flips and without leaving the class of cylindrical arrangements. 

\begin{restatable}{theorem}{thmMainCipa}\label{thm:main_cipa}
The triangle flip graph of intersecting cylindrical arrangements of $n$ pseudocircles is connected. 
\end{restatable}

Showing that every intersecting arrangement $\mathcal{A}$ can be flipped into some cylindrical arrangement then completes the proof of \Cref{thm:main_ipa}.
We further study the diameter of the resulting flip graphs. 
For both, the triangle flip graph of intersecting arrangements and
for the triangle flip graph on intersecting cylindrical arrangements, 
we obtain asymptotically tight bounds for the diameter.

\begin{restatable}{proposition}{propDiamCipa}\label{prop:diam_cipa}
The triangle flip graph of intersecting cylindrical arrangements of $n$ pseudocircles has diameter at least $2\binom{n}{3}$ and at most $4\binom{n}{3}$.
\end{restatable}

\begin{restatable}{proposition}{propDiamIpa}\label{prop:diam_ipa}
The triangle flip graph of intersecting arrangements of $n$ pseudocircles has diameter~$\Theta(n^3)$. 
\end{restatable}

\paragraph{Outline.}
In \Cref{sec:prelim}, we introduce concepts and terminology more formally and present some key preliminaries. 
In \Cref{sec:plineflipgraphs} we prove the $(n-2)$-connectivity of the triangle flip graph of arrangements of~$n$ pseudolines.
Additionally we prove the $(n-2)$-connectivity of the triangle flip graph of arrangements of~$n$ lines where in fact the transitions between arrangements can be realized by continuous motions of lines.
\Cref{sec:pcircleflipgraphs} is dedicated to the triangle flip graphs of intersecting (cylindrical) arrangements of pseudocircles.
Our studies also lead us to new characterizations of cylindrical arrangements, which we present in \Cref{sec:cylindrical}.
We conclude the article with a discussion of implications, future research directions, and open problems in \Cref{sec:conclusion}.

%%%%%%%%%%%%%%%%%%%%%%%%%%%%%%%%%%%%%%%%%%%%%
%%%%%%%%%%%%%%%%%%%%%%%%%%%%%%%%%%%%%%%%%%%%%
\section{Preliminaries}\label{sec:prelim}
%%%%%%%%%%%%%%%%%%%%%%%%%%%%%%%%%%%%%%%%%%%%%
%%%%%%%%%%%%%%%%%%%%%%%%%%%%%%%%%%%%%%%%%%%%%

A pseudoline partitions the plane into two (unbounded) regions, while a pseudocircle $C$ partitions the plane into a bounded region, the interior $\interior{C}$, and an unbounded region, the exterior~$\exterior{C}$.
An arrangement of curves induces a decomposition of the plane into \defi{vertices} (the intersection points), \defi{edges} (maximal contiguous vertex-free pieces of the curves), and \defi{cells} (connected components of the plane after removing all curves)\footnote{In general there may also be endpoints of curves.}.
One can think of an arrangement as a plane graph-like object consisting of vertices, edges and cells. 

In this work we only consider simple intersecting arrangements.
Hence, every vertex in the arrangements has degree 4 and every cell is incident to at least two pseudocircles respectively
three pseudolines.
A bounded cell with $k$ edges along its boundary is a \defi{$k$-cell}, a 3-cell is a \defi{triangle}, and a 2-cell is a \defi{digon} (a.k.a.~\emph{empty lens}).

\paragraph{Arrangements of pseudolines.} 
An arrangement of $n$ pseudolines in the Euclidean plane has $2n$ unbounded cells. Typically, Euclidean
arrangements are assumed to be \defi{marked}, that is, 
one of the unbounded cells is designated as \defi{north-cell}
$z_N$ and pseudolines are oriented such that the north-cell is in their left halfplane. The \defi{south-cell} $z_S$ of a marked arrangement
is the unique unbounded cell separated from $z_N$ by all pseudolines. 
Unless stated otherwise, all pseudoline arrangements in this paper are marked.
The \defi{canonical labeling} of a marked arrangement
is the numbering of the pseudolines in the order of visit when 
traversing the arrangement from 
the north-cell to the south-cell via unbounded cells, while leaving all vertices of the arrangement to the right.
\Cref{fig:example_marked_arrangement} shows an example.
We consider two arrangements of pseudolines to be \defi{isomorphic} 
if the induced cell decompositions are isomorphic and the marking is preserved\footnote{Isomorphism can also become subtle: There are arrangements of lines with a disconnected realization space, i.e., there are combinatorial isomorphisms with no corresponding homotopy \cite{Mnev1988,Suvorov1988}.}, 
i.e., if there is a bijection between the pseudolines 
 that induces incidence and orientation preserving bijections of vertices, edges, and cells.

\begin{figure}[htb]
	\centering
	\includegraphics[scale=0.75]{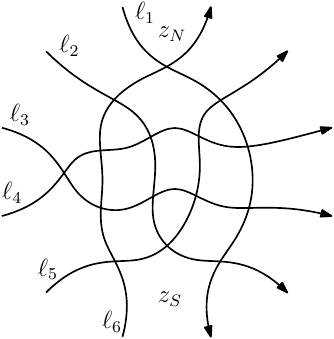}
	\caption{A marked arrangement and its canonical labelling.}
	\label{fig:example_marked_arrangement}
\end{figure}

Given an arrangement $\AA$ of $n$ pseudolines we say that 
the triple $\ell_i, \ell_j, \ell_k$ has \defi{positive orientation} if the crossing 
of the pseudolines $\ell_i$ and $\ell_k$ is in the northern halfplane 
of the pseudoline $\ell_j$, otherwise the triple has \defi{negative orientation}.
Assuming the canonical labeling and 
ordered triples $i < j < k$ we obtain a mapping 
$\sigma_\AA: \binom{[n]}{3} \to \{+,-\}$. This is the 
3-signotope associated with $\AA$. It is known (see~\cite{FelsnerWeil2001}) 
that 3-signotopes encode isomorphism classes of 
arrangements of pseudolines. In the proof of 
\Cref{thm:connectivity} we will use triple orientations to 
argue that certain paths in the flip graph of arrangements 
are vertex-disjoint. 

A pseudoline $\ell$ of an arrangement $\AA$ of pseudolines 
is \defi{extreme} if all crossings lie in the northern halfplane (above~$\ell$) or all crossings lie in the southern halfplane
(below~$\ell$). 
An arrangement $\AA$ of pseudolines 
is \defi{shellable} if it contains an extreme pseudoline $\ell$
and $\AA \arrminus \ell$ is shellable or empty, where $\AA \arrminus \ell$ denotes the arrangement of all pseudolines of $\AA$ except $\ell$.
The sequence obtained by iteratively removing extreme pseudolines from a
shellable arrangement 
is called a \defi{shelling sequence}. 
For example, the arrangement in \Cref{fig:example_marked_arrangement} is shellable with shelling 
sequence $\ell_1, \ell_5, \ell_2, \ell_3, \ell_4, \ell_6$.

\paragraph{Arrangements of pseudocircles.}

We consider two arrangements of pseudocircles to be \defi{isomorphic} 
if the induced cell decompositions are isomorphic, i.e., if there is a bijection between the pseudocircles that induces incidence  preserving bijections of vertices, edges, and cells, while also keeping the ccw-orientations of pseudocircles.
In the plane, there are exactly four arrangements of three pairwise intersecting pseudocircles up to isomorphism (see \Cref{fig:fourTypes}). 
Following~\cite{FelsnerScheucher2019}, we call the 
arrangement with 8 triangles depicted in \Cref{fig:fourTypes}(a) the \defi{Krupp} 
arrangement and denote the other ones as \defi{NonKrupp}.
We refer to the NonKrupp arrangement whose unbounded cell has complexity~$k$ as
NonKrupp$(k)$,
e.g., NonKrupp$(2)$ is the arrangement shown in \Cref{fig:fourTypes}{(c)}.

\begin{figure}[htb]
\centering
\includegraphics{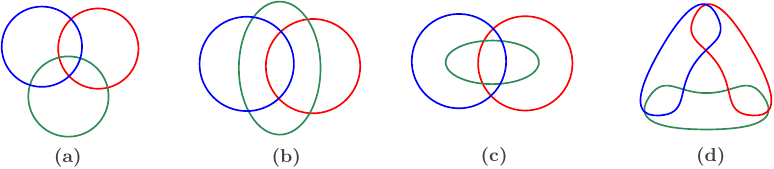}
\caption{The four non-isomorphic arrangements of 3 pairwise intersecting pseudocircles in the plane. 
	(a)~shows the Krupp and (b)--(d) show the three types of NonKrupp arrangements.}
\label{fig:fourTypes}
\end{figure}

Recall that an arrangement of pseudocircles is \emph{cylindrical} if 
there is a point~$p$ which is contained in the interiors of all the pseudocircles. 
Among the arrangements of  \Cref{fig:fourTypes}, the arrangement in \Cref{fig:fourTypes}{(d)}, i.e., the 
NonKrupp$(3)$, is the only non-cylindrical one.

Some related work considers arrangements of pseudocircles on the sphere, e.g.~\cite{FelsnerScheucher2019}.
In this setting there are only two non-isomorphic intersecting arrangements of three 
pseudocircles: The arrangements 
{(b)},  {(c)}, and {(d)} of \Cref{fig:fourTypes} 
are isomorphic on the sphere, since they only differ in the choice of
the unbounded cell. 
In this article we consider arrangements in the plane.
 
A subclass of cylindrical arrangements are the great-pseudocircle arrangements. 
With the above terminology, a great-pseudocircle arrangement is an
arrangement of pseudocircles with the property that every triple 
of pseudocircles from the arrangement induces a Krupp subarrangement.

\paragraph{Sweeps and flips.}
Sweeping is an algorithm design paradigm in  computational geometry. 
The idea is to explore a space with a lower dimensional sweep surface in a way that the continuous exploration can be discretized 
via certain \emph{events} and hence becomes algorithmically handlable. 
For arrangements of curves in the plane, 
the sweeping curve can also be seen as part of the arrangement: 
A \defi{sweep} of an arrangement $\AA$ of pseudolines is 
a transformation of a pseudoline $c$ within
$\AA$ such that one side of $c$ is increasing and collects additional crossings of~$\AA$.
An event of the sweep occurs whenever $c$ passes a crossing, or, equivalently, when a triangle of $\AA$ is flipped.
Snoeyink and Hershberger~\cite{SnoeyinkHershberger1991} 
studied sweeps of arrangements of 
pseudocircles with a pseudocircle.  
They showed that a sweep of an arrangement $\CC$ with a pseudocircle $C$ of $\CC$ is possible in two phases.
In the \emph{growing phase}, $C$ explores the exterior of $C$ in $\CC$ and in the \emph{shrinking phase} it explores the interior of $C$ in $\CC$.
In both phases the events of the sweep can be restricted to one of the following three possibilities, see \Cref{fig:flips}. 
\defi{Digon-create}: 
a pseudocircle gains two intersections with a pseudocircle;
\defi{digon-collapse}:
a pseudocircle loses its two intersections with a pseudocircle;
\defi{triangle~flip}:
a pseudocircle (or pseudoline) moves over the crossing of two others.

\begin{figure}[htb]
    \centering
    \includegraphics[scale=0.9]{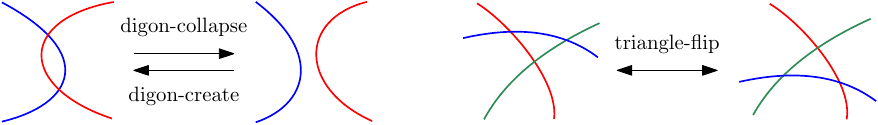}
    \caption{
    An illustration of the three flip operations. 
    }
    \label{fig:flips}
\end{figure}

Snoeyink and Hershberger actually 
prove the following more general sweeping lemma 
for families of simple curves that pairwise intersect at most twice 
and are either bi-infinite or closed. 

\begin{lemma}[{\cite[Lemma 3.2]{SnoeyinkHershberger1991}}]
\label{lem:snoeyink_hershberger}
Let $\AA$ be an arrangement of pseudocircles and bi-infinite curves that pairwise intersect at most twice. Then $\AA$ can be swept starting from any curve $C$ in $\AA$ by 
using the three operations triangle flip, digon-create, and digon-collapse, while maintaining the 
property that any two curves intersect at most twice.
\end{lemma}

It will be convenient to have a separate sweeping lemma for \emph{lenses}. 
A \defi{lens} in an arrangement is a bounded region in a subarrangement formed by two intersecting pseudocircles.
An \defi{arc} is a contiguous subset of a pseudocircle, starting and ending at a vertex of the arrangement. 
Let $\AA$ be an arrangement of pseudocircles and 
let $Q$ be (the closure of) a lens bounded by two arcs $L$ and~$R$.

An \defi{arc of $Q$} is an arc of a pseudocircle~$C$ which has both 
endpoints on the boundary of $Q$
and whose relative interior lies in the interior of $Q$.
If an arc $a$ of $Q$ has both endpoints on $L$ or both endpoints on $R$, then $a$ forms a lens with $L$ or~$R$, respectively. 
Otherwise the arc has one endpoint on $L$ and one on $R$, in this case we
call the arc \defi{transversal} (see \Cref{fig:zipping}).

\begin{figure}[htb]
	\centering
	\includegraphics[page=3,scale=0.95]{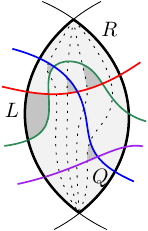}\\
	\caption{Illustration of \Cref{lemma:zipping}: a~lens~$Q$ with four transversal arcs and a sweep of~$Q$. }
	\label{fig:zipping}
\end{figure}

The following lemma is a consequence of~\Cref{lem:snoeyink_hershberger}.
For self-containment, we include a direct~proof.

\begin{lemma}
\label{lemma:zipping}
Let $Q$ be a lens bounded by the arcs $L$ and $R$. If 
all the arcs of $Q$ are transversal, 
then, using only triangle flips, $L$ can be swept towards~$R$ until the interior of $Q$ does not contain a vertex of the arrangement. 
\end{lemma}

\begin{proof}
We orient all arcs of $Q$ from $L$ to~$R$,
and consider the arrangement $\AA|_Q$ of these arcs. 
\begin{claim}\label{claim:A}
If two oriented arcs $a$ and $a'$ intersect in two points $p$ 
and $q$ such that $p$ precedes~$q$ on $a$, then $p$ 
precedes $q$ on $a'$.
\end{claim}
\begin{proof}
\renewcommand{\qedsymbol}{$\vartriangleleft$}
Suppose for the contrary that $q$ is the first intersection with $a$ on $a'$.
Consider the region $T$ bounded by the subarcs of $a$ and $a'$ from $L$ to $p$, respectively, and the corresponding part of $L$; see \Cref{fig:newzipping}(left). Then the arc $a'$ enters the interior of $T$ at $q$, but cannot leave it anymore to reach~$R$, a contradiction.
\end{proof}

\begin{claim}\label{claim:B}
The directed graph of the arrangement $\AA|_Q$ is acyclic.
\end{claim}

\begin{proof}
\renewcommand{\qedsymbol}{$\vartriangleleft$}
Suppose $\AA|_Q$ contains a directed cycle. Let $Z$ be a directed cycle such that the enclosed area is minimal. Then $Z$ is the boundary of a cell of the arrangement.
Let $v_1, e_1, v_2, \ldots e_k, v_{k+1}=v_1$ be the sequence of vertices
and edges along $Z$. From \Cref{claim:A} we know that $k\geq 3$ holds. 
Let edge $e_i$ be part of arc $a_i$ and let
$\alpha_i$ be the subarc of $a_i$ starting at $L$ and ending in $v_i$. Similarly, let $\beta_i$ be the subarc of $a_i$ starting in $v_{i+1}$ and ending at $R$.

We will show by induction that for all $j=2,\ldots,k$ there is an intersection
of $\alpha_j$ and $\beta_1$. 
This implies \Cref{claim:B} because the intersection of $\alpha_k$ and $\beta_1$ is after $v_1$ on $a_1$ and before $v_1$ on $a_{k}$,
a contradiction to \Cref{claim:A}.

For $j=2$, the intersection of 
$\alpha_2$ and $\beta_1$ is at $v_2$. For the induction step, 
given that there is an intersection of $\alpha_{j-1}$ 
and $\beta_1$ we consider the region~$T$ in~$Q$ whose boundary consists of the directed subpath 
$S$ from $v_2$ to $v_j$ on the boundary of~$Z$ together with 
$\beta_1$, $\beta_{j-1}$, and some piece of $R$, or in case 
$\beta_1$ and $\beta_{j-1}$ meet in a point~$q$ the boundary consists of $S$, the arc from $v_2$ to $q$ on $\beta_1$
and the arc from $v_j$ to $q$ on $\beta_{j-1}$, 
see~\Cref{fig:newzipping}(right). It is important to
note that due to \Cref{claim:A} the interior of $Z$
is not part of $T$.
The final part of $\alpha_j$ is inside $T$ and ends at~$v_j$.
Since another intersection of $\alpha_j$ and~$\beta_{j-1}$ is 
forbidden by \Cref{claim:A} it follows that $\alpha_j$ had to 
enter region $T$ through~$\beta_1$.
\end{proof}

\begin{figure}[htb]
	\centering   
	\hbox{}
	\hfill
	\includegraphics[page=1,scale=0.95]{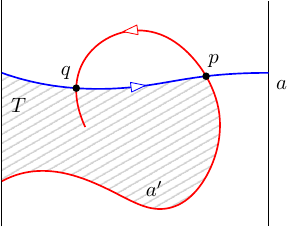}
	\hfill
	\includegraphics[page=2,scale=0.95]{figs/newzipping.pdf}
	\hfill
	\hbox{}
	
	\caption{Illustration of the proof of \Cref{lemma:zipping}.}
	\label{fig:newzipping}
\end{figure}

Since the directed graph  $\AA|_Q$ 
is acyclic, there is at least one source. 
Every source gives rise to a triangle with an edge on~$L$
which can be flipped. A topological ordering of 
of $\AA|_Q$ yields a sequence of triangle flips which
can be used to sweep $Q$ with $L$.
This completes the proof of \Cref{lemma:zipping}.
\end{proof}

Let us briefly remark how \Cref{lemma:zipping} can be deduced from \Cref{lem:snoeyink_hershberger}:
Extend the
left boundary curve $L$ of $Q$  and all the arcs of $Q$ to 
bi-infinite curves. 
For the extension of each of the curves 
use two pseudorays -- one at each end. 
The collection of all pseudorays 
can be chosen such that they are pairwise disjoint (non-crossing) and do not introduce any further crossings.
A sweep with sweep-line $L$ in the resulting arrangement of bi-infinite curves then yields the desired sequence of triangle flips. 

It is also possible to add the extending pseudorays in such a way 
that each pair of the resulting bi-infinite curves has an odd 
number of crossings and infer \Cref{lemma:zipping}
from the odd-crossing sweeping lemma of Bokowski et 
al.~\cite[Lemma~5.2]{BokowskiKPZ18}.

%%%%%%%%%%%%%%%%%%%%%%%%%%%%%%%%%%%%%%%%%%%%%%%%%%
%%%%%%%%%%%%%%%%%%%%%%%%%%%%%%%%%%%%%%%%%%%%%%%%%%
\section{Flip Graphs on Arrangements of Pseudolines}
\label{sec:plineflipgraphs} 
%%%%%%%%%%%%%%%%%%%%%%%%%%%%%%%%%%%%%%%%%%%%%%%%%%
%%%%%%%%%%%%%%%%%%%%%%%%%%%%%%%%%%%%%%%%%%%%%%%%%%

In this section we prove~\Cref{thm:connectivity} and \Cref{prop:conn-real}.
We first show that shellable arrangements have favourable properties, in particular they belong to polytopal clusters in the flip graph $\mathbf{F}_n$. In \Cref{ssec:connect-lines} we use these insights to show $(n-2)$-connectivity of the triangle flip graph $\mathbf{FL}_n$ of arrangements of lines -- even under continuous transformations --, i.e., \Cref{prop:conn-real}.
In \Cref{ssec:connect-pslines} we define good sets of triangles and show how to use them in connection with shellable arrangements to prove~\Cref{thm:connectivity}.

\begin{lemma}
\label{lemma:shell_lines}
Let $\SS$ be a shellable arrangement with shelling sequence $s_{\sigma(1)},\ldots,s_{\sigma(n)}$ and 
let $\lambda_1 < \lambda_2 < \ldots < \lambda_n$ be a sequence of real numbers.
Then $\SS$ can be realized as an arrangement of lines where for each $i$ 
the line $s_i$ has slope~$\lambda_i$.
\end{lemma}

\begin{proof}
  Start with a line $s_{\sigma(n)}$ of slope $\lambda_{\sigma(n)}$.
  When  $s_{\sigma(n)},\ldots,s_{\sigma(k+1)}$ have been fixed, then
  insert $s_{\sigma(k)}$ as a line of slope $\lambda_{\sigma(k)}$ such that
  all the crossings of the previously inserted lines are above or below~$s_{\sigma(k)}$.
\end{proof}

Let $\Lambda=(\lambda_1,\ldots,\lambda_n)$ be a slope vector with $\lambda_1 < \ldots < \lambda_n$
and let $\mathbf{A}_\Lambda$ be the set of simple line arrangements with $n$ lines
which can be realized with the slopes 
of $\Lambda$. 
Moreover, let $\mathbf{F}_\Lambda$ be the restriction of the flip graph $\mathbf{F}_n$ on arrangements of
$n$ pseudolines to the elements of $\mathbf{A}_\Lambda$.
The following theorem has been shown by Felsner and Ziegler (see \Cref{fig:zonotope} for an illustration and~\cite{FeZi01} for definitions).

\begin{figure}[htb]
\centering
\includegraphics[trim={3.5cm 6.6cm 3.5cm 6.6cm},clip,width=0.25\textwidth]{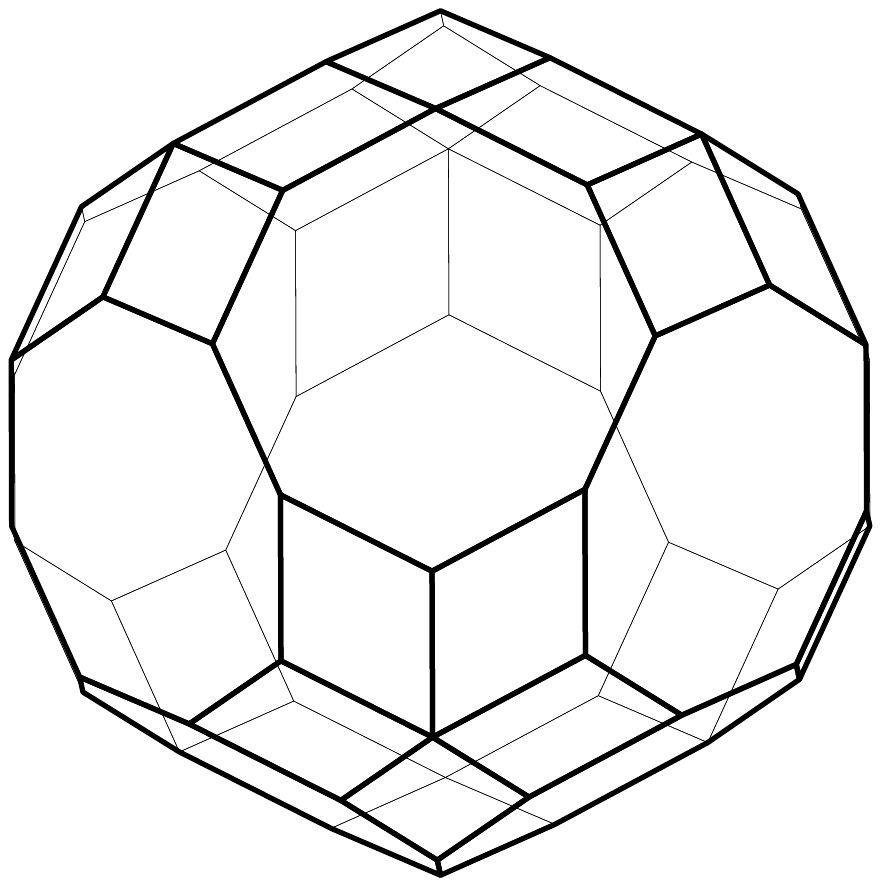}
\caption{The graph $\mathbf{F}_5$, drawn as the skeleton of a $3$ dimensional zonotope. Figure taken from \cite{FeZi01}.}
\label{fig:zonotope}
\end{figure}

\begin{theorem}[{\cite{FeZi01}}]\label{thm:zonotope}
  The graph $\mathbf{F}_\Lambda$ is the skeleton graph of an $(n-2)$-dimensional
  zonotope.
\end{theorem}

Balinski's theorem \cite{balinki61} asserts that the skeleton graph of a \mbox{$d$-dimensional} polytope is $d$-connected. 
Since zonotopes are special polytopes we have the following:

\begin{corollary}\label{cor:connectivity-Lambda}
  The graph $\mathbf{F}_\Lambda$ is $(n-2)$-connected.
\end{corollary}

The following lemma will be essential in Section~\ref{ssec:connect-pslines}  when we define paths in the flip graph
which connect an arrangement~$\AA$ to an element of $\mathbf{A}_\Lambda$.

\begin{lemma}\label{lemma:shell_trans}
Let $\AA$ be an arrangement of pseudolines 
and let $\SS$ be a shellable arrangement. Then there
is a path $P$ from $\AA$ to $\SS$ in the flip graph $\mathbf{F}_n$
such that on $P$ every triple of pseudolines is flipped at most~once.
\end{lemma}

\begin{proof}
  Let $\ell_1,\ldots,\ell_n$ be the pseudolines of $\AA$ 
  and let
  $s_{\sigma(1)},\ldots,s_{\sigma(n)}$ be a shelling sequence of~$\SS$. 
  Line
  $s_{\sigma(1)}$ is extreme in $\SS$.  If $s_{\sigma(1)}$ is above/below the
  crossings of $\SS$ we can sweep the line~$\ell_{\sigma(1)}$
  upwards/downwards in $\AA$ until all crossings lie below/above the line, see
  Figure~\ref{fig:proof_lemma1}. A flip-path from $\AA\arrminus \ell_{\sigma(1)}$ to
  $\SS\arrminus s_{\sigma(1)}$ can be constructed recursively.  This
  recursively constructed path has no flip of a triple involving
  $\ell_{\sigma(1)}$ while all the triples of the attached final subpath
  involve~$\ell_{\sigma(1)}$. Hence, no triple is flipped twice.
\end{proof}

\begin{figure} [htb]
\centering
\includegraphics[width=0.7\textwidth]{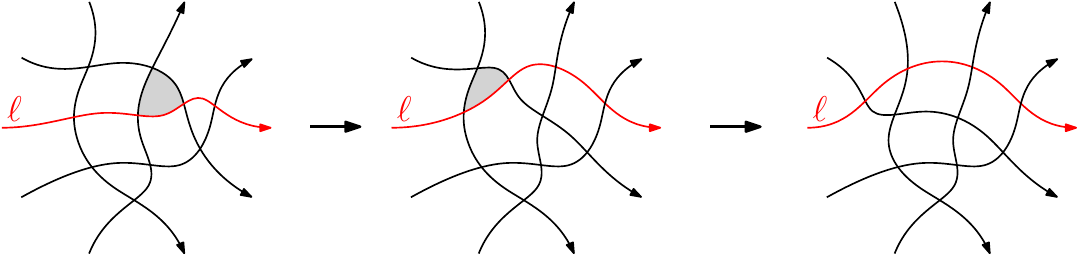}
\caption{The pseudoline $\ell$ is swept upwards with a sequence of triangle flips. In the final position all crossings are below $\ell$.}
\label{fig:proof_lemma1}
\end{figure}

\subsection{Proof of \texorpdfstring{\Cref{prop:conn-real}}{Proposition~\ref{prop:conn-real}}: \texorpdfstring{$(n-2)$}{(n-2)}-connectivity for line arrangements}
\label{ssec:connect-lines}

\begin{restatable}{lemma}{lemmaconnrealspace}
\label{lemma:connrealspace}
 The realization space of a shellable arrangement $\mathcal{S}$ is path-connected.
\end{restatable}
 \begin{proof}
 We proof the lemma by induction over the number of lines in $\mathcal{S}$. The statement is true for the arrangement consisting of a single line, here the realization space is $\mathbb{R}^2$ (the arrangement is marked, hence, the vertical line is not allowed).

 Assume that $\ell$ is an extreme line of $\mathcal{S}$. Let $\mathcal{S}_1$ and $\mathcal{S}_2$ be two realizations of $\mathcal{S}$.
 Rotate $\mathcal{S}_1$ until $\ell$ has the same slope in $\mathcal{S}_1$ and~$\mathcal{S}_2$. 
 It can happen that the rotation changes the combinatorial type of~$\mathcal{S}_1$, by moving a line over the north-pole. To prevent this, we apply an affine transformation, by continuously stretching the plane in $x$-direction. 
 This transformation changes the slopes in~$\mathcal{S}_1$ and~$\mathcal{S}_2$ arbitrarily close to zero. 
 If the slopes are sufficiently close to zero, we can do the aforementioned rotation without changing the combinatorial type of $\mathcal{S}_1$.
 After the rotation, we can translate $\mathcal{S}_1$, such that $\ell$ is equal in $\mathcal{S}_1$ and $\mathcal{S}_2$. 
 
 By induction, there is a continuous motion of lines, transforming $\mathcal{S}_1 \arrminus \ell$ into $\mathcal{S}_2 \arrminus \ell$, without changing combinatorial types.
 It is possible that in this motion, a crossing of $\mathcal{S}_1 \arrminus \ell$ moves over $\ell$, which would change the combinatorial type of $\mathcal{S}$.
 To prevent this, we can move $\ell$ in $\mathcal{S}_1$ and $\mathcal{S}_2$ far enough downwards or upwards.
 \end{proof}
\begin{restatable}{lemma}{lemmafliprealization}
\label{lemma:flipsarerealizable}
 Let $\mathcal{L}$ and $\mathcal{L}'$ be two line arrangements, realized with the same slope vector $\Lambda$, who have the same combinatorial type or are adjacent in $\mathbf{FL}_{n}$. Then $\mathcal{L}$ can be transformed into $\mathcal{L}'$, only by continuous motions of lines, such that all intermediate arrangements are also realized with $\Lambda$ and are combinatorially equivalent to either $\mathcal{L}$ or $\mathcal{L}'$.
\end{restatable}
 \begin{proof}
 Embed the two line arrangements in the $\{z=0\}$ respectively $\{z=1\}$ plane in $\mathbf{R}^3$.  
 Since they are both realized with the same slope set, we can interpolate the $i$-th line in $\mathcal{L}$ and $\mathcal{L}'$ by a plane~$H_i$. 
 Intersecting the plane arrangement $H=\{ H_i \mid  i = 1 \dots n \}$ with the $\{z=t\}$ planes with $t$ increasing from $0$ to $1$ yields a continuous motion of line arrangements with slopes in $\Lambda$, transforming $\mathcal{L}$ into $\mathcal{L}'$.
 
 Since a triple of the planes $\{ H_i \mid  i = 1 \dots n \}$ can cross at most once, all intermediate arrangements have the combinatorial type of $\mathcal{L}$ or $\mathcal{L}'$. 
 \end{proof}

\begin{proof}[Proof of \Cref{prop:conn-real}]
Let $\AA$ and $\AA'$ be two realizable arrangements, realizable with slope vectors $\Lambda$ and~$\Lambda'$, respectively.
We will only consider the subgraph of $\mathbf{FL}_n$ induced by the arrangements that are in $\mathbf{F}_\Lambda$ or  $\mathbf{F}_{\Lambda'}$. After deleting $n-3$ arrangements in this graph, each of $\mathbf{F}_\Lambda$ and $\mathbf{F}_{\Lambda'}$ is still connected, by Corollary~\ref{cor:connectivity-Lambda}. They also share at least one shellable arrangement $\mathcal{S}$,
by Lemma~\ref{lemma:shell_lines}. It follows that there is path from $\AA$ to $\mathcal{S}$ and from $\AA'$ to $\mathcal{S}$. 

We can choose these paths to lie completely in $\mathbf{F}_\Lambda$ and $\mathbf{F}_\Lambda'$, respectively. 
We realize the first path as a continuous motion of lines from $\AA$ to $\mathcal{S}$, both realized with slopes in $\Lambda$, by repeatedly applying Lemma~\ref{lemma:flipsarerealizable}. We can do the same for $\AA'$ and $\mathcal{S}$, both realized by slopes in $\Lambda'$. 

We obtain two realizations $\mathcal{S}_1$ and $\mathcal{S}_2$ of $\mathcal{S}$. They can be transformed into each other with a continuous motion of lines, by Lemma~\ref{lemma:connrealspace}.
\end{proof}

\subsection{Proof of \texorpdfstring{\Cref{thm:connectivity}}{Proposition~\ref{thm:connectivity}}: \texorpdfstring{$(n-2)$}{(n-2)}-connectivity for pseudoline arrangements}
\label{ssec:connect-pslines}

A set $T$ of triples of~$[n]$ 
is \defi{good} if for every assignment
$\alpha:T \to \{+,-\}$ there exists a shellable arrangement
$\SS_\alpha$ such that in $\SS_\alpha$ the orientation of
the triple of pseudolines corresponding to $t \in T$
is $\alpha(t)$.
First, we use shellable arrangements and the concept of good triples to
prove the statement of Theorem~\ref{thm:connectivity} for all $n \geq 43$.
Later, in \Cref{sec:43_to_24}, we perform the step from $n \geq 43$ to $n \geq 24$.

\begin{lemma}
\label{lemma:many_t}
Every arrangement $\AA$ on $n \ge 3$ pseudolines 
admits a sequence of distinct triangles 
$t_1,\ldots,t_k$
 and a sequence of distinct pseudolines $\ell_{\tau(1)},\ldots,\ell_{\tau(k)}$
with $k \ge \frac{n}{3}$
such that the first triple 
in $t_1,\ldots,t_k$
that involves $\ell_{\tau(i)}$ is the triple $t_i$.
\end{lemma}

\begin{proof}
  The sequence of triangles is constructed one by one.
  Suppose that~$t_1,\ldots,t_{j-1}$ have been constructed. If there is a pseudoline
  $\ell$ which is not incident to any of these triangles.
  The sweeping lemma implies that 
  there exists a triangle $t$ incident to~$\ell$. 
  Since $\ell$ is not incident to the previous triangles $t_1, \ldots, t_{j-1}$, the triangle $t$ is distinct.
  Set $t_j=t$ and $\ell_{\tau(j)} = \ell$.  
  Since each $t_j$ covers
  three pseudolines the construction yields a sequence of length~$k \geq \frac{n}{3}$.
\end{proof}
 
\begin{lemma} \label{lemma:good_t}
Let $t_1,\ldots,t_k$ be
a sequence of distinct triangles and $\ell_{\tau(1)},\ldots,\ell_{\tau(k)}$ be a sequence of distinct pseudolines
such that the first triple in $t_1,\ldots,t_k$ involving $\ell_{\tau(i)}$ is the triple~$t_i$.
Then $T := \{t_1,\ldots,t_k\}$ is a good set of triples.
\end{lemma}

\begin{proof} 
  To show that $T$ is good we first extend $\tau$
  to a permutation, i.e., we define appropriate values of $\tau(j)$ for $j=k+1,\ldots,n$.
  Now consider an arbitrary $\alpha:T \to \{+,-\}$. We construct
  a shellable arrangement $\SS_\alpha$ with shelling sequence
  $\ell_{\tau(k)},\ldots,\ell_{\tau(1)},\ell_{\tau(n)},\ldots,\ell_{\tau(k+1)}$.
  The lines $\ell_{\tau(k+1)},\ldots,\ell_{\tau(n)}$ are inserted one by one
  in extreme position -- above or below -- in the initially empty
  arrangement. When it comes to inserting $\ell_{\tau(j)}$ for $j \le k$,
  then due to the construction rule the lines $\ell_{\tau(k)},\ldots,\ell_{\tau(j+1)}$ 
  do not belong to triangle $t_j$. Hence,  $\ell_{\tau(j)}$ is the missing line
  of $t_j$ and depending on~$\alpha(t_j)$ the line~$\ell_{\tau(j)}$ can be inserted
  above or below all crossings. This yields the shellable arrangement~$\SS_\alpha$
  with the required~properties.\relax
\end{proof}

To prove Theorem~\ref{thm:connectivity} we will show that for every arrangement
$\AA$ of $n$ pseudolines there are $n-2$ internally disjoint paths
in $\mathbf{F}_n$ which connect $\AA$ to $n-2$ distinct
shellable arrangements. We also fix an arbitrary set $\Lambda$ of slopes and 
consider the flip graph~$\mathbf{F}_\Lambda$ of line arrangements with slopes
in $\Lambda$.  Shellable arrangements are vertices of this $(n-2)$-connected flip graph
$\mathbf{F}_\Lambda$ (Lemma~\ref{lemma:shell_lines} and Corollary~\ref{cor:connectivity-Lambda}).
The $(n-2)$-connectivity of the full flip graph $\mathbf{F}_n$ is then implied by 
Lemma~\ref{lem:connectivity-lemma}.

For a given arrangement $\AA$ of $n$ pseudolines we will define 
shellable arrangements $\SS_1,\SS_2,\ldots,\SS_{n-2}$ and construct internally disjoint paths from $\AA$ to
the $\SS_i$. The construction is based on a fixed good set of triangles
$T=\{t_1,\ldots,t_k\}$ (Lemma~\ref{lemma:many_t} and~\ref{lemma:good_t}) and a
selection $t_{k+1},\ldots,t_{n-2}$ of additional triangles of $\AA$. 

We say that a set $D$ of triangles of $\AA$ is \defi{compatible} if there exists an
arrangement $\AA[D]$ such that the orientation of a triple~$t$ of pseudolines
in $\AA$ and $\AA[D]$ agrees if and only if $t \not\in D$. 
For each $i\in[n-2]$ we define a path~$P_i$ in $\mathbf{F}_n$. 
The path~$P_i$ depends on a 2-element subset $T_i$ of $T$ with the property that $\{t_i\} \cup T_i$ is compatible.
It starts in 
$\AA$ and moves in two or three steps to the  
arrangement $\AA_i = \AA[\{t_i\} \cup T_i]$.

For each $i\in[n-2]$, 
let $\alpha_i:T \to \{+,-\}$ record the orientation of triangles of $T$ in
$\AA_i$ and let $\SS_{i}$ be a shellable arrangement which agrees with
$\AA_i$ on the orientations of triples in the good set~$T$, where the definition of good implies that such an $\SS_{i}$ exists.
With $Q_i$ we denote the path in~$\mathbf{F}_n$ leading
from~$\AA_i$ to $\SS_{i}$ as described in
Lemma~\ref{lemma:shell_trans}.  Since $\AA_i$ and $\SS_{i}$ agree on
the orientations of $T$ and on~$Q_i$ every triple is flipped at most once, all
the arrangements of the path~$Q_i$ respect~$\alpha_i$. Hence, if
$\alpha_i \neq \alpha_j$ the paths $Q_i$ and~$Q_j$ will be disjoint.  
In the
case $i \leq k$ the set $T_i$ will consist of~$t_i$ and some second triangle
$t_{\rho(i)}$ and the path $P_i$ begins with
$\AA,\AA[t_i],\AA[t_i,t_{\rho(i)}] =\AA_i$ and continues with~$Q_i$.
In the
case $i > k$ the set $T_i$ will consist of two triangles
$t_{\rho_1(i)},t_{\rho_2(i)}$ and the path $P_i$ begins with
$\AA,\AA[t_i],\AA[t_i,t_{\rho_1(i)}],\AA[t_i,t_{\rho_1(i)},t_{\rho_2(i)}] =\AA_i$
and continues with $Q_i$.

\begin{figure}[htb]
\centering
\includegraphics[width=0.9\textwidth]{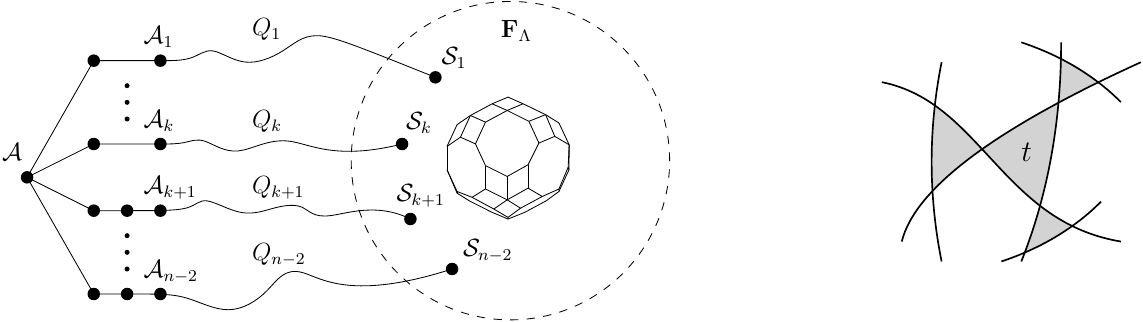}
\caption{\emph{Left:} Illustration for the construction of the paths in the proof of Theorem~\ref{thm:connectivity}. \emph{Right:} Illustration for \Cref{obs:degree}.}
\label{fig:path_plan}
\end{figure}

We will construct the pairs $T_i$ for $i=1,\ldots n-2$ such that $T_i \neq T_j$ for $i\neq j$. 
Together with the compatibility of $T_i=\{t_i,t_{\rho(i)}\}$ for $i \leq k$ and the triples
$\{t_i\}\cup T_i = \{t_i,t_{\rho_1(i)},t_{\rho_2(i)}\}$ for $i>k$ 
this implies that $\alpha_i\neq \alpha_j$, whence the corresponding path family is disjoint except at their common start~$\AA$.

\begin{lemma}
\label{obs:degree}
A triangle $t$ of an arrangement $\AA$ is incompatible with at most three other
triangles of~$\AA$.
\end{lemma}

\begin{proof}
  Two distinct triangles of an arrangement are compatible if after flipping the first
  of them the second is still a triangle, i.e., if they share no vertex.
  A triangle can share a vertex with at most three other triangles; see \Cref{fig:path_plan} (right).
\end{proof}
The following lemma will allow us to deduce the degree of connectivity from the disjoint paths we will~construct.
\begin{lemma}\label{lem:connectivity-lemma}
    Let $G$ be a graph and let $H$ be an $r$-connected subgraph of $G$. If every vertex~$v$ of $G$ 
    can be connected to $r$ distinct vertices of $H$ with a family of $r$ internally disjoint paths, 
    then $G$ is $r$-connected.
\end{lemma}

\begin{proof} 
    Let $U$ be a set of vertices of $G$ with $|U| < r$. We claim that $G\setminus U$ is connected.
    Consider two vertices $v,v'$ in $G\setminus U$. Since $v$ is connected to $H$ by $r$ 
    disjoint paths and $|U| < r$ there is at least one path $p$ in $G\setminus U$ from $v$ to a vertex $w\in H$.
    Likewise there is a path $p'$ in $G\setminus U$ from $v'$ to a vertex $w'\in H$. Since $H$ is $r$-connected 
    there is a path $q$ from~$w$ to~$w'$ in $H\setminus U$. The concatenation $p,q,p'$ is a walk from $v$ to 
    $v'$ in  $G\setminus U$ and hence, $U$ is not separating.
\end{proof}

We now have all ingredients to prove~\Cref{thm:connectivity} for 
$n \geq 43$. Additional ideas needed to obtain the stated result, i.e., 
the validity in the range $n\leq 8$ and $n\geq 24$,
are deferred to~\Cref{sec:43_to_24}.

\begin{lemma}
\label{lem:plineconnectivity}
For $n\geq 43$, the triangle flip graph $\mathbf{F}_n$ on marked arrangements of $n$ pseudolines 
is $(n-2)$-connected. 
\end{lemma}

\begin{proof}
Let $T=\{t_1,\ldots,t_k\}$ be a good set of triangles as described above (Lemma~\ref{lemma:many_t} and~\ref{lemma:good_t}). Then it suffices to show that for $n$ sufficiently large and $k \geq \frac{n}{3}$ it is possible to find $n-2$ different two element subsets $T_i$ of $T=\{t_1,\ldots,t_k\}$ such that
for each $i$ the set $T_i\cup \{t_i\}$ is compatible and for each $i$ with
$i \leq k$ the triangle $t_i$ actually belongs to $T_i$.

For $i \leq k$ define $\rho(i)$ as the first index in the circular order on
$1,\ldots,k$ such that $t_i,t_{\rho(i)}$ is a compatible pair of
triangles. The circular distance of $i$ and $\rho(i)$ is at most four,
therefore, if $k\geq 9$ the pairs are all distinct. To be sure that $k\geq 9$
we need $n\geq 27$.

Recall that for $i > k$ the triangle $t_i$ is an additional triangle not in $T$. We define the pairs $T_i$ one after the other. First we let
$C_i\subset T$ be the set of good triangles compatible with $t_i$. By \Cref{obs:degree} it holds that
$|C_i| \geq k-3$. Since each triangle in~$C_i$ is compatible with at
least $|C_i|-4\geq k-7$ others from this set (again by \Cref{obs:degree}), there are at least
$\frac{1}{2}(k-3)(k-7)$ compatible (unordered) pairs in~$C_i$. 
At most $n-3$ of them have
been assigned as a $T_j$ for some $j$.  Hence, if
$\frac{1}{2}(k-3)(k-7) > n-3$ we find a pair $T_i$ with the required
properties.  Since $k \geq \frac{n}{3}$ it is sufficient to have $n \geq 43$.
\end{proof}

%%%%%%%%%%%%%%%%%%%%%%%%%%%%%%%%%%%%%%%%%%%%%%%%%%%%%%%%%%
\subsubsection{Finalizing the proof of Theorem~\ref{thm:connectivity}} 
\label{sec:43_to_24}
%%%%%%%%%%%%%%%%%%%%%%%%%%%%%%%%%%%%%%%%%%%%%%%%%%%%%%%%%%

So far we gave a proof of the statement for $n\geq 43$. Next we argue for the missing cases of~$n$.

\paragraph{\texorpdfstring{$n\leq 7$}{n at most 7}: Realizability:}
It is known that all projective arrangements with at most 8
pseudolines can be realized with lines, see \cite{GoodmanPollack90}.
By adding a line at infinity, we can embed an Euclidean arrangement into a projective arrangement. This operation does not change realizability. This implies that all Euclidean arrangement with $n\leq7$ pseudolines can be realized.
Hence for $n\leq7$ we have $\mathbf{F}_n=\mathbf{FL}_n$
and the connectivity of $\mathbf{F}_n$ follows from \Cref{prop:conn-real}.

\paragraph{\texorpdfstring{$n\geq 34$}{n at most 34}: Larger good sets of triangles:}

In \Cref{ssec:connect-pslines}
we charged three conflicting triangles in~$T$ to each triangle $t_i$ for $i=1,\ldots,n-2$.
We now show that we only have to account for two conflicting~triangles.

First we make sure that every triangle in $T$ is incompatible with at most two other
triangles in~$T$. To achieve this we choose $t_j$ as the first triangle visited when
traversing $\ell_{\tau(j)}$ from left to right. 
A triangle which is in conflict with three others is not leftmost on any of the three incident lines, see \Cref{fig:path_plan}.

When it comes to define the additional triangles
$t_{k+1},\ldots,t_{n-2}$ we identify the set $B$ of all triangles in $\AA$ which are
incompatible to three triangles from $T$. Let $\AA' = \AA[B]$ be the arrangement
obtained by flipping all the triangles in $B$. Flipping a triangle of $B$ destroys
its three incompatible triangles, these belong to $T$. We claim that at most $k - |B|$
out of the $k$ triangles of $T$ are triangles of $\AA'$. Think of a bipartite graph
with sides $B$ and $T$. Each element of $B$ has three neighbors in $T$ while each 
element of $T$ has at most three neighbors in $B$, hence, there are at least $|B|$
elements in~$T$ which have a neighbor in $B$.
Since $\AA'$ has at least $n-2$ triangles and every triangle of $\AA'$ is a triangle of $\AA$
we can avoid picking a triangle of $B$ when we chooses $t_{k+1},\ldots,t_{n-2}$.

The crucial inequality now becomes $\frac{1}{2}(k-2)(k-5) > n-3$
and we only need $n\geq 35$. Using that $k \ge \frac{n}{3}$ is an integer
we see that already $n \ge 34$ is sufficient.

\paragraph{\texorpdfstring{$n\geq 24$}{n at most 24}: Even larger good sets of triangles:}

From Lemma~\ref{lemma:good_t} we know that a sequence $t_1,\ldots,t_k$ of triangles is a good set if
there is a sequence $\ell_{\tau(1)},\ldots,\ell_{\tau(k)}$ of pseudolines such that the first triple in the
sequence of triangles which is incident to $\ell_{\tau(i)}$ is the triangle $t_i$.

In Lemma~\ref{lemma:many_t} we constructed a sequence of triangles and charged three lines to
each triangle in the sequence. To reduce this number we consider the bipartite graph $H_\AA$
whose vertices are the triangles and the pseudolines of $\AA$ and 
edges correspond to incidence between triangles and pseudolines. Let $K$ be a component of $H_\AA$
We construct one by one a sequence of triangles from~$K$. The first triangle $t_1$ is chosen arbitrarily
and its three incident lines are marked as \emph{visited}. Suppose that~$t_1,\ldots,t_{j-1}$ have been constructed
and all the lines incident are marked visited. If there is a pseudoline 
 in $K$ which is is not visited, then we find a triangle $t_j$ in $K$ which is incident
 to a visited and an unvisited line. Let $\ell_{\tau(i)}$ be an unvisited
line incident to $t_j$ and mark the lines incident to $t_j$ visited. 
If there are $s$ lines in~$K$, then the procedure constructs a sequence of at least $\lceil \frac{s-1}{2} \rceil$
triangles.
The sequences of triangles constructed for different components of $H_\AA$ can be 
concatenated arbitrarily to form a sequence of good triangles. If there are $c$ components with an odd
number of lines in $H_\AA$, then this yields a set 
of $T$ of good triangles of size at least $\frac{n-c}{2}$.

We will show below that there are no components with 3 or 5 lines. This implies that odd components contain at 
least 7 lines, whence there are at most $\lfloor \frac{n}{7} \rfloor$ odd components and we find a set of
at least $k=\lceil \frac{3n}{7} \rceil$ triangles. Reviewing the argument from the previous section we
first see that $k\geq 9$ is achieved whenever $n\geq 21$. A triangle in the good set $T$ may by construction have up 
to three incompatible triangles in $T$. For the additional triangles we can again argue that 
they have at most two incompatible triangles in $T$. The condition for finding the appropriate pairs 
$T_i$ now becomes $\frac{1}{2}(k-2)(k-6) > n-3$. With $k=\lceil \frac{3n}{7} \rceil$ we find that
$n \ge 24$ is sufficient.

It remains to show that for $n>5$ there are no components with
3 or 5 lines in $H_\AA$. From the sweeping lemma we know that a line which is incident to only one triangle 
is extremal. Assuming that there is a component with three lines $\ell_1,\ell_2,\ell_3$ and one
triangle~$t$ the extremality of the lines implies that 
for each $i$ all the crossings of $\AA$ are on the side of~$\ell_i$ which contains $t$. 
Hence, all the crossings are in the triangle $t$ which requires $n=3$, a~contradiction.

For the discussion of components with 5 lines we look at the arrangement $\BB$ induced by the  five lines in~$\AA$. 
In general a cell of $\BB$ can be a union of cells of $\AA$. However,
if $t$ is a triangle in~$\BB$ and there is a line~$\ell$ incident to~$t$ and~$t$ is the only 
incident triangle on this side of~$\ell$ in $\BB$, then $t$ is a triangle in~$\AA$:
The side of $\ell$ containing~$t$ contains crossings, hence,~$\ell$ can be swept towards this side 
in $\AA$. Since $\BB$ is a component the first triangle taken by the sweep must belong to $\BB$ and by 
assumption $t$ is the unique choice.
Now consider the possible arrangements $\BB$ of five lines. 
Suppose first that there is a line $\ell$ which is extremal in $\BB$. Since $\BB$ is a component, $\ell$ can only 
be swept in one direction in $\AA$, hence,~$\ell$ is extremal in $\AA$. Consider $\AA\arrminus\{\ell\}$, this 
arrangement has the component $\BB\arrminus\{\ell\}$ with only four lines. In this component we find an extremal line $\ell'$ which must also be extremal in~$\AA\arrminus\{\ell\}$.
Removing this line we end up with an arrangement
$\AA\arrminus\{\ell,\ell'\}$ which has a component with three lines. This contradicts our previous insight.

If $\BB$ has no extremal line, then $\BB$ is a star arrangement consisting of a pentagonal cell
surrounded by five triangles (this can be checked from Fig.~8 in \cite{FelsnerWeil2001}). Take any line of $\BB$ and sweep it away from the pentagon in~$\AA$,
this is equivalent to flipping a star triangle.
It is easily verified that the flip is not generating new triangles, hence, in $H_\AA$ it only
affects the component of $\BB$. Since after the flip $\BB$ has an extremal line we can complete the 
argument as in the previous case.

\begin{remark*}
It is conjectured that every non-monochromatic bicoloring of the pseudolines of an arrangement
with at least three lines  yields a bichromatic triangle, see e.g.\ \cite[Chapter~6]{BjoenerLVWSZ1993} or \cite{fps-aapl-20}.
The bichromatic triangle conjecture can be restated in terms of $H_\AA$, here the conjecture is that the graph is connected.
If the conjecture was true, we could construct a good set of triangles of size~$\frac{n-1}{2}$ in every arrangement of $n$ pseudolines.
This would allow to decrease the bound in \Cref{thm:connectivity} to~$n\geq 22$.
\end{remark*}

%%%%%%%%%%%%%%%%%%%%%%%%%%%%%%%%%%%%%%%%%%%%%%%%%%
%%%%%%%%%%%%%%%%%%%%%%%%%%%%%%%%%%%%%%%%%%%%%%%%%%
\section{Flip Graphs on Arrangements of Pseudocircles}
\label{sec:pcircleflipgraphs}
%%%%%%%%%%%%%%%%%%%%%%%%%%%%%%%%%%%%%%%%%%%%%%%%%%
%%%%%%%%%%%%%%%%%%%%%%%%%%%%%%%%%%%%%%%%%%%%%%%%%%

The outline of this section is as follows.
In \Cref{sec:main_ipa} we prove that every arrangement of pairwise intersecting pseudocircles can be transformed into a cylindrical arrangement by flips.
In \Cref{sec:main_cipa} we then show that the flip graph of cylindrical arrangements of pairwise intersecting pseudocircles is connected,
and hence the flip connectivity carries over to arrangements of pairwise intersecting pseudocircles.
\Cref{sec:diam_cipa,sec:distance} are dedicated to proving asymptotically tight bounds on the diameters of these flip graphs.

%%%%%%%%%%%%%%%%%%%%%%%%%%%%%%%%%%%%%%%%%%%%%%%%%%
%%%%%%%%%%%%%%%%%%%%%%%%%%%%%%%%%%%%%%%%%%%%%%%%%%
\subsection{Proof of \texorpdfstring{\Cref{thm:main_ipa}}{Theorem~\ref{thm:main_ipa}}: Connectivity Intersecting}
\label{sec:main_ipa}
%%%%%%%%%%%%%%%%%%%%%%%%%%%%%%%%%%%%%%%%%%%%%%%%%%
%%%%%%%%%%%%%%%%%%%%%%%%%%%%%%%%%%%%%%%%%%%%%%%%%%

For an arrangement~$\AA$ of $n$ pairwise intersecting pseudocircles,
we fix some point $p$ and then iteratively expand pseudocircles by triangle flips until all pseudocircles contain~$p$, i.e., until the arrangement is cylindrical. 
The flip-connectivity then follows from \Cref{thm:main_cipa}.

For convenience we introduce the notation $\AA - \{C_1, \ldots, C_k\}$ to denote the arrangement that is obtained from~$\AA$ by removing the pseudocircles $\{C_1, \ldots, C_k\}$.
We say that two pseudocircles~$C$ and~$C'$ are \defi{parallel} in $\AA$ if every vertex of $\AA - \{C,C'\}$ lies in  $(\interior{C} \cap \interior{C'})
\cup (\exterior{C} \cap \exterior{C'})$.

\thmMainIpa*

\begin{proof}
Let $\AA$ be an intersecting arrangement of $n$ pseudocircles. We show by induction on~$n$ that $\AA$~can be transformed into a cylindrical arrangement with a finite number of triangle flips. 
The statement of the theorem then follows from \Cref{thm:main_cipa}.

The induction base is trivially fulfilled for $n=2$. 
For the induction step,
we choose $p$ as a point which lies inside a maximum number of pseudocircles.
If $p$ lies in the interior of all pseudocircles, then $\AA$ is already cylindrical and we are done.

Hence, we may assume that there exists a pseudocircle $C$ which does not contain~$p$ in its interior.
We show how to expand $C$ using only triangle flips until it contains~$p$.
First observe that \Cref{lem:snoeyink_hershberger} guarantees the existence of
a flip to expand $C$.
Since~$C$ already intersects all other pseudocircles, 
this must be a triangle or digon-collapse flip. 
As long as there exists a triangle flip expanding $C$, 
we perform it and transform $\AA$ accordingly.

Suppose now that $C$ does not yet contain $p$ and can only be expanded by collapsing a digon~$D$ formed with another pseudocircle~$C'$. Note that we may assume $p$ to not lie in $D$ (and in no other cell adjacent to $C$), since otherwise we could expand $C$ until it contains $p$ without performing any flip.
Furthermore, every other pseudocircle $C''$ intersects the lens $\interior{C} \cap \exterior{C'}$ transversally as otherwise $C''$ could not intersect both, $C$ and~$C'$. 
More specifically, every other pseudocircle $C''$ intersects the lens $\interior{C} \cap \exterior{C'}$ in two arcs, and each arc has one endpoint on~$C$ and one on~$C'$; see  \Cref{fig:sketch_theorem3} for an illustration.
Now we can use \Cref{lemma:zipping} to expand~$C'$ until $C$ and~$C'$ are parallel.

\begin{figure}[htb]
\centering
\includegraphics[scale=0.7, page=1]{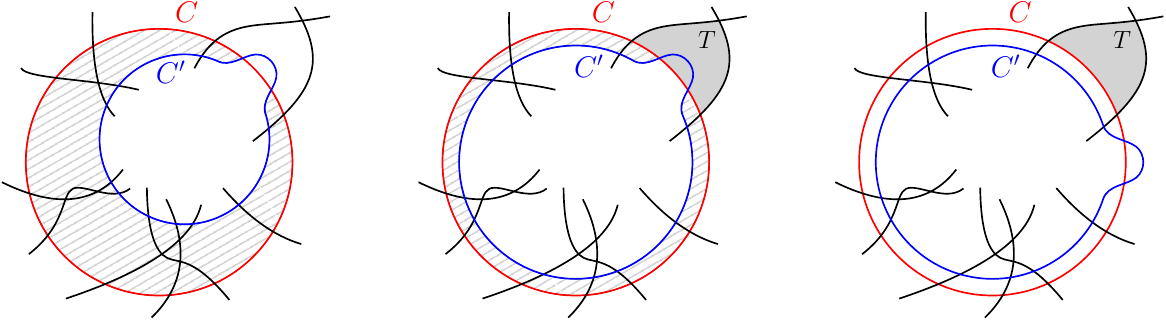}
\caption{
\emph{Left:}~$C$ forms a digon with $C'$ that is in $\exterior{C} \cap \interior{C'}$. 
\emph{Middle:}~flip $C'$ so that it becomes parallel to~$C$.
\emph{Right:}~flip $C'$ so that $T$ becomes a triangle in $\AA$.
}
\label{fig:sketch_theorem3}
\end{figure}

Next consider the arrangement $\AA':=\AA-\{C'\}$. 
By the induction hypothesis, 
$\AA'$ can be transformed into a cylindrical arrangement by a finite sequence of triangle flips.
We now carefully mimic this flip sequence on~$\AA$, while maintaining that $C$ and $C'$ are parallel.
Suppose that a triangle~$T$ in $\AA'$ is flipped.
If none of the edges of~$T$ belongs to~$C$, 
we can directly apply this triangle flip also in~$\AA$. 
If one of the edges $e$ of~$T$ belongs to $C$ and $e$ is crossed by~$C'$, then the digon~$D$ is located along $e$.
In this case, we apply two triangle flips to $C'$ so that the digon is transferred to one of the two neighboring edges of $C$ as illustrated in \Cref{fig:sketch_theorem3}, 
obtaining that $e$ is not crossed by~$C'$ (without changing~$\AA'$).
Finally, if $e$ is not crossed by~$C'$, then we apply the according triangle flip twice, namely, once for $C$ and once for~$C'$.
This concludes the proof.
\end{proof}

%%%%%%%%%%%%%%%%%%%%%%%%%%%%%%%%%%%%%%%%%%%%%%%%%%
%%%%%%%%%%%%%%%%%%%%%%%%%%%%%%%%%%%%%%%%%%%%%%%%%%
\subsection{Proof of \texorpdfstring{\Cref{thm:main_cipa}}{Theorem~\ref{thm:main_cipa}}: Connectivity Cylindrical}
\label{sec:main_cipa}
%%%%%%%%%%%%%%%%%%%%%%%%%%%%%%%%%%%%%%%%%%%%%%%%%%
%%%%%%%%%%%%%%%%%%%%%%%%%%%%%%%%%%%%%%%%%%%%%%%%%%

In this section we prove flip-connectivity for cylindrical intersecting arrangements by showing that any given arrangement can be flipped to a \emph{canonical} arrangement, which we define via an arrangement of pseudoparabolas. An \defi{arrangement of pseudoparabolas} is a finite collection of $x$-monotone curves defined over a common interval such that every two curves are either disjoint or intersect in two
points where the curves cross. 
The following proposition gives a reversible mapping from 
cylindrical arrangements of pseudocircles to arrangements of pseudoparabolas. 
The result has been announced by Bultena et 
al.~\cite[Lemma 1.3]{BultenaGruenbaumRuskey1999} and a full proof has 
been given by Agarwal et~al.~\cite[Lemma~2.11]{AgarwalNPPSS2004}.

\begin{restatable}[\cite{BultenaGruenbaumRuskey1999,AgarwalNPPSS2004}]
{proposition}{lemCylinder}\label{prop:cylinder}
A cylindrical arrangement of pseudocircles $\CC$ can be mapped to an arrangement of pseudoparabolas $\AA$ in an axis-aligned rectangle $B$ such that $\CC$ is isomorphic to the arrangement obtained by identifying the two 
vertical sides of $B$ and mapping the resulting cylindrical surface homeomorphically to a ring in the plane.
\end{restatable}

\noindent
We define a canonical intersecting arrangement $\BB_n^-$ of pseudoparabolas as follows:

\begin{enumerate}
\item On the left, the pseudoparabolas are labelled $C_1, \ldots, C_n$ from top to bottom and on the right, the top to bottom order is the same.
\item For every $1 \leq i \leq n$, the pseudoparabola $C_i$ intersects the other pseudoparabolas first in increasing and then in decreasing order: $C_1, \ldots, C_n, C_n, \ldots, C_1$.
\end{enumerate}

Analogously, $\BB_n^+$ denotes the intersecting arrangement of pseudoparabolas in which $C_i$ intersects the other 
pseudoparabolas first in decreasing and then in increasing 
order. Closing the pseudoparabolas of $\BB_n^-$ and $\BB_n^+$ above yields cylindrical intersecting arrangements of 
pseudocircles, which we call $\AA_n^-$ and $\AA_n^+$, respectively; see \Cref{fig:canonical_wiring_diagram} for an illustration. Note that every triple of pseudocircles forms a NonKrupp$(2)$ in $\AA_n^-$ and a NonKrupp$(4)$ in~$\AA_n^+$.

In the following we show that every cylindrical arrangement can be flipped to both
canonical arrangements, $\AA_n^-$ and~$\AA_n^+$. While one of them is sufficient to prove connectivity (\Cref{thm:main_cipa}), 
both will play an important role to obtain tight bounds on the diameter (\Cref{prop:diam_cipa}).

\begin{figure}[bht]
\centering
\includegraphics[scale=0.7, page=1]{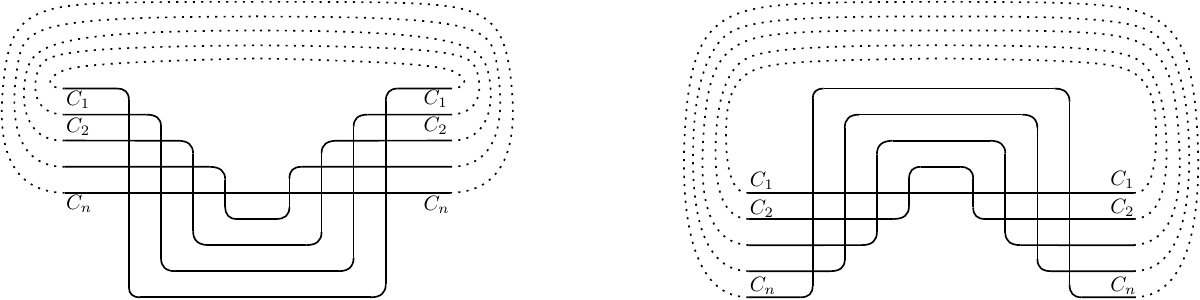}
\caption{
The two canonical arrangements $\AA_n^-$ (left) and $\AA_n^+$ (right).
}
\label{fig:canonical_wiring_diagram}
\end{figure}

\thmMainCipa*

\begin{proof}
Let $\AA$ be an intersecting, cylindrical arrangement of $n$ pseudocircles. Then, using \Cref{prop:cylinder}, we can represent $\AA$ as a pseudoparabola arrangement, in which we label the curves from top to bottom by $C_1, \ldots, C_n$. The idea is to flip the pseudoparabolas downwards one by one in the order of  increasing indices. When $C_n$ has been processed the arrangement~$\BB_n^-$ corresponding to~$\AA_n^-$ is reached. When flipping $C_k$ downwards the following invariant holds (see \Cref{fig:canonical_wiring_diagram2} for an~illustration).

\begin{figure}[tb]
\centering
\includegraphics[page=2,scale=0.9]{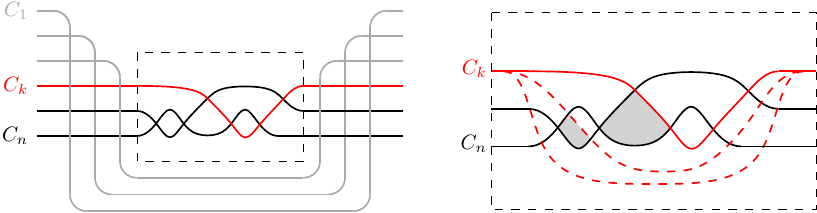}
\caption{An illustration of the proof of \Cref{thm:main_cipa}: flipping the pseudoparabola $C_k$ downwards.
}
\label{fig:canonical_wiring_diagram2}
\end{figure}

\begin{enumerate}
    \item For every $ 1 \le i < k$, the pseudoparabola $C_i$ intersects the other pseudoparabolas 
    in the same order as in~$\BB_n^-$:
    first in increasing and then in decreasing order, that is, 
    $C_1, \ldots, C_n, C_n, \ldots, C_1$.
    
    \item For every $i \geq k$, the pseudoparabola $C_i$ first intersects $C_1, \ldots , C_{k-1}$ in increasing order, then it has the intersections with the pseudoparabolas $C_j$ for $j\in \{k,\ldots,n\}\setminus\{i\}$ in some order, and finally intersects $C_{k-1}, \ldots, C_1$ in decreasing order.
\end{enumerate}

Let $\AA_{\ge k}$ denote the subarrangement induced by $C_k, \ldots, C_n$.
We claim that in $\AA_{\ge k}$ the pseudoparabola $C_k$ can be swept downwards using only triangle flips. Applying \Cref{lem:snoeyink_hershberger} to the subarrangement $\AA_{\ge k}$, we can sweep $C_k$ downwards using a sequence of triangle, digon-create and digon-collapse flips. However, the latter two flips cannot occur: Since every pair of pseudoparabolas already intersects, a digon-create is impossible. 
Assume that a digon-collapse happens, then there is a pseudoparabola $C$ which forms a digon $D$ with $C_k$ 
such that $D$ lies below~$C_k$. The arc of $C$ belonging to $D$ lies below $C_k$ and the complement (in particular the two unbounded arcs of~$C$) lie above~$C_k$. 
This is not possible since $C_k$ is the topmost pseudoparabola in $\AA_{\ge k}$.

By the invariant, all crossings of $\AA_{\ge k}$ lie above all the pseudoparabolas $C_1, \ldots, C_{k-1}$. Therefore, we can perform the flip sequence from $\AA_{\ge k}$ also in the original arrangement~$\AA$. Since we do not flip triangles involving $C_1, \ldots, C_{k-1}$, their intersection orders remain unchanged. Moreover, since in the subarrangement $\AA_{\ge k}$ all crossings between $C_{k+1}, \ldots, C_n$ lie above $C_k$, the order of intersections in~$\AA_{\ge k}$ on $C_k$ is $C_{k+1}, \ldots, C_n, C_n, \ldots, C_{k+1}$. From this it follows that the invariant has been preserved.

When $C_n$ has been processed the invariant implies that the canonical arrangement $\BB_n^-$ has been reached. This completes the proof of \Cref{thm:main_cipa}.
\end{proof}

%%%%%%%%%%%%%%%%%%%%%%%%%%%%%%%%%%%%%%%%%%%%%%%%%%
%%%%%%%%%%%%%%%%%%%%%%%%%%%%%%%%%%%%%%%%%%%%%%%%%%
\subsection{Proof of \texorpdfstring{\Cref{prop:diam_cipa}}{Proposition~\ref{prop:diam_cipa}}: Diameter Cylindrical}\label{sec:diam_cipa}
%%%%%%%%%%%%%%%%%%%%%%%%%%%%%%%%%%%%%%%%%%%%%%%%%%
%%%%%%%%%%%%%%%%%%%%%%%%%%%%%%%%%%%%%%%%%%%%%%%%%%

\propDiamCipa*

\begin{proof}  
Concerning the upper bound,
following the proof of \Cref{thm:main_cipa}, 
we can transform any cylindrical arrangement~$\AA$ of $n$ pairwise intersecting pseudocircles 
into the canonical arrangement~$\AA_n^-$. 
To do so, we iteratively transform the selected pseudocircle $C_i$
by only flipping downwards
until all crossings of $C_j$ and $C_k$ with $j,k > i$ 
are above $C_i$.
For every $C_i$, there are at most $2\binom{n-i}{2}$ such crossings below $C_i$ and hence, in total, we need at most 
\[
  \sum_{i=1}^n 2\binom{n-i}{2} = 2\sum_{i=2}^{n-1} \binom{i}{2} = 2\binom{n}{3}
\]
flips to reach~$\AA_n^-$.
Hence, we need at most $4\binom{n}{3}$ flips to transform any two arrangements into each~other.

Concerning the lower bound, 
note that $2\binom{n}{3}$ flips are needed to transform $\AA_n^-$ into $\AA_n^+$ because every NonKrupp$(2)$ 
needs to be flipped twice to become a NonKrupp$(4)$.
\end{proof}

%%%%%%%%%%%%%%%%%%%%%%%%%%%%%%%%%%%%%%%%%%%%%%%%%%
%%%%%%%%%%%%%%%%%%%%%%%%%%%%%%%%%%%%%%%%%%%%%%%%%%
\subsection{Proof of \texorpdfstring{\Cref{prop:diam_ipa}}{Proposition~\ref{prop:diam_ipa}}: Diameter Intersecting}
\label{sec:distance}
%%%%%%%%%%%%%%%%%%%%%%%%%%%%%%%%%%%%%%%%%%%%%%%%%%
%%%%%%%%%%%%%%%%%%%%%%%%%%%%%%%%%%%%%%%%%%%%%%%%%%
 
Extending the notion of parallelism from Section~\ref{sec:main_ipa}, we call a subset $\CC$ of pseudocircles from an intersecting arrangement~$\AA$ \defi{parallel}
if every vertex of the arrangement $\AA \arrminus \CC$ either lies in the exterior of every $C \in \CC$ or in the interior of every $C \in \CC$. 
A set of parallel pseudocircles is also called a \defi{bundle}. 
Note that, for any bundle~$\CC$ of size $k$ in $\AA$, the order of intersections of the $n-k$ pseudocircles from $\AA - \CC$ is the same for every $C \in \CC$. Hence, $\AA \arrminus \CC$ splits $\CC$ into~$2(n-k)$ \defi{parts}.
Moreover, the crossings within the bundle can easily be moved between the parts using only triangle flips; see \Cref{fig:partition}.

\propDiamIpa*
\begin{proof}
In the proof of \Cref{thm:main_ipa} showed that any two intersecting arrangement of $n$ pseudocircles can be transformed into each other via finitely many triangle flips. 
Next we show that $O(n^3)$ flips are always sufficient. 
We then conclude the cubic diameter by showing that $\Omega(n^3)$ flips are sometimes necessary.

Let $\AA$ be an intersecting arrangement of $n$ pseudocircles.
The general situation 
is that we have a point $p$ and a bundle $\CC_{[k]} = \{C_1,\ldots,C_k\}$ of pseudocircles not yet containing $p$, which are simultaneously expanded to eventually contain~$p$. 
In order to guarantee the cubic bound on the diameter, we need to carefully distribute and move the crossings formed by the bundle in relation to crossings of other pseudocircles with the bundle.

\paragraph{Initial setup.} 
We start with an arbitrary point $p$ from the plane and initialize $\DD$ as the set of pseudocircles that already contain $p$ in the interior. 
We choose $C_1$ as an arbitrary pseudocircle which is not yet in~$\DD$,
set the bundle size to $k=1$, set $\CC_{[1]} := \{C_1\}$, and proceed with the following~steps.

\paragraph{Step~1 (Reset if bundle becomes too large).}

If $ k > \frac{n}{2}$,
we reset the process by choosing a new 
$p$ as an interior point of the bundle 
$\CC_{[k]}$ 
and
set $\DD := \CC_{[k]}$.
Next, we  choose $C_1$ as an arbitrary pseudocircle from $\AA \arrminus \DD$,
set the bundle size to $k:=1$, and set $\CC_{[1]} := \{C_1\}$.
Note that such a reset happens at most once, because,
as soon as $|\DD| > \frac{n}{2}$ holds,
the bundle size~$k$ cannot exceed $\frac{n}{2}$ again.

\begin{figure}[htb]
\centering
\includegraphics[scale=0.75]{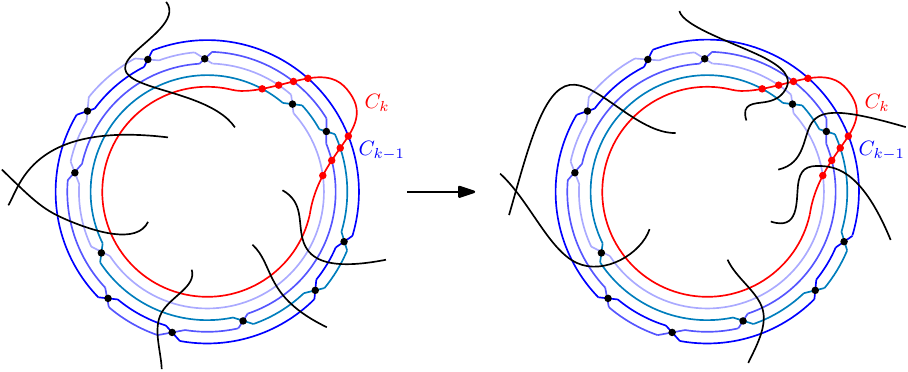}
\caption{An illustration of the distribution of crossings in step~2.
\emph{Left:}~Well-distributed crossings w.r.t.\ to the bundle $C_1, \ldots, C_{k-1}$. 
\emph{Right:}~Well-distributed crossings w.r.t.\ to the bundle $C_1, \ldots, C_{k}$. 
}

\label{fig:partition}
\end{figure}

\paragraph{Step~2 (Distribute crossings).}
If $k \ge 2$ then the $k$ pseudocircles 
in the bundle $\CC_{[k]}$ form~$2\binom{k}{2}$ crossings.
Also, $\CC_{[k]}$ is split
by the $n-k$ pseudocircles from 
$\AA \arrminus \CC_{[k]}$ 
into  $2(n-k)$ parts.
The aim is to distribute the crossings of $\CC_{[k]}$ such that 
each part contains at most $\frac{k^2}{n}$ crossings, where $\frac{k^2}{n} \ge \frac{2\binom{k}{2}}{2(n-k)}$ due to $k \le \frac{n}{2}$.
In this case we say that the crossings are \defi{well-distributed}.

We now describe how to maintain the property that crossings 
are well-distributed when a new pseudocircle~$C_k$ enters the bundle $\CC_{[k-1]}$.
The $2\binom{k-1}{2}$ crossings formed by $\CC_{[k-1]}$ 
are well-distributed among the parts of $\CC_{[k-1]}$ 
that are induced by $\AA \arrminus \CC_{[k-1]}$,
i.e., 
each part contains at most $\frac{(k-1)^2}{n}$ crossings.
When $C_k$ enters the bundle,
two boundaries between parts of size at most $\frac{(k-1)^2}{n}$
disappear and on the old boundaries, $2k-2$ crossings involving~$C_k$ appear, see \Cref{fig:partition}. To move crossings
from parts of the expanded bundle $\CC_{[k]}$ which contain 
too many crossings into sparser parts of it we use triangle flips.

The cost for moving an excess of one crossing to a part where it complies
with the bound is at most $n-k$, i.e., at most $n-k$ triangle flips
are needed.
Hence, to reach well-distribution 
for the expanded bundle $\CC_{[k]}$, we need at most
\[
\left(2\frac{(k-1)^2}{n}+2k-2\right)\cdot (n-k) \ < \ \left(2\frac{k^2}{n}+2k\right)\cdot n \ \le \ 3 kn
\]
triangle flips, where we used $k \leq \frac{n}{2}$ for the last inequality.

\paragraph{Step~3 (Expand bundle).}
Let $\AA' = \AA - \CC_{[k-1]}$. 
Using \Cref{lem:snoeyink_hershberger}, 
we expand $C_k$~in~$\AA'$ via a sequence of $x_k$ triangle flips until $C_k$ either contains $p$ or forms a digon in $\AA'$ with another pseudocircle, which we denote by~$C_{k+1}$. 
We next show how to mimic this flip sequence in~$\AA$.

For a triangle $T$ flipped in the sequence in~$\AA'$, denote by $C_k,D,E$ the three involved pseudocircles.
While in~$\AA'$, the pseudocircle $C_k$ moves over the crossing of $D$ and $E$,
in $\AA$ we have to move the entire bundle $\CC_{[k]}$ over this crossing.
Note that the pseudocircles $D$ and~$E$ cross $C_k$ consecutively along $C_k$ 
in $\AA'$. Hence they bound a part $P$ of~$\CC_{[k]}$, which contains at 
most $\frac{k^2}{n}$ crossings of~$\CC_{[k]}$.
The strategy is to transfer 
all crossings of the bundle from the part $P$ to
one of its neighboring parts.
As each crossing can be transferred  by a single triangle flip, 
the total transfer requires at most $\frac{k^2}{n}$ flips.
\Cref{fig:partition_move} gives an~illustration.

\begin{figure}[htb]
\centering
\includegraphics[scale=0.75]{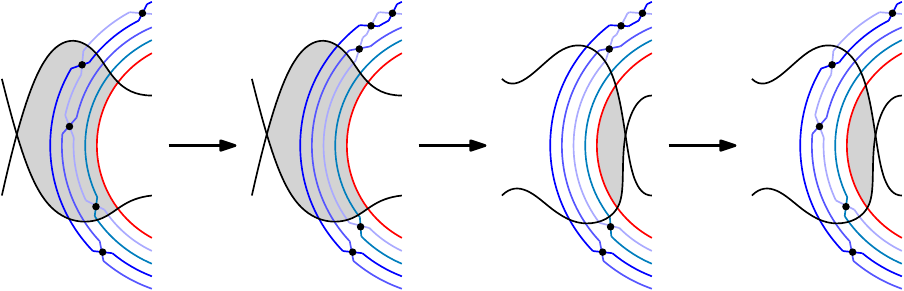}
\caption{An illustration of the expanding step. 
}

\label{fig:partition_move}
\end{figure}

Next we move the pseudocircles $C_1,\ldots,C_k$ from $\CC_{[k]}$ 
over the crossing of $D$ and~$E$. 
More precisely, we flip the triangles determined by $(C_i,D,E)$ for all~$i \in \{1,\ldots,k\}$.
After performing 
these $k$ flips, we move all previously transferred 
crossings back to ensure that each part of $\CC_{[k]}$ again contains at most $\frac{k^2}{n}$ crossings, which also requires at most $\frac{k^2}{n}$ flips.

A single flip of $C_k$ in $\AA'$ thus maps to a sequence of at most 
$(k+2\frac{k^2}{n}) \le 2 k$ triangle flips of $\CC_{[k]}$ in $\AA$. 
We call this sequence a 
\defi{bundle flip}.
Since we had 
$x_k$ triangle flips when expanding $C_k$ in $\AA'$, we get a total
of at most $2kx_k$ triangle flips to move 
$\CC_{[k]}$
in~$\AA$.

If the process 
terminates with $C_k$ containing $p$ in the interior, then also the other pseudocircles $C_1,\ldots,C_{k-1}$ from the bundle contain $p$ in the interior. 
In this case we add $\CC_{[k]}$ to~$\DD$.
Then we pick another pseudocircle $C_1$ which is not yet in $\DD$ arbitrarily, set $k=1$, and continue with Step~1.
Also, if $p$ is in a cell which is incident to~$C_k$ in $\AA'$,
then we can move $p$ to the adjacent cell inside $C_k$. This move
can also be done in $\AA$ where $p$ is moved across the whole bundle.
In this case again we add the bundle to $\DD$, pick a new $C_1$ and 
continue with Step~1.
Otherwise we proceed with Step~4.

\paragraph{Step~4 (Add $C_{k+1}$ to bundle).}
If $C_k$ does not yet contain $p$ and cannot be further expanded
by triangle flips, $C_k$ forms a digon with another
pseudocircle $C_{k+1}$ from $\AA \arrminus \CC_{[k]}$. 
Since the digon lies outside of $C_k$ and is incident to $C_k$, 
we know that $p$ is not in the digon (otherwise we would have 
moved $p$ to the other side of $C_k$). 
Consider again the arrangement $\AA \arrminus \CC_{[k-1]}$.
Using \Cref{lemma:zipping}, we can expand $C_{k+1}$ in $\AA \arrminus \CC_{[k-1]}$ via triangle flips until it is parallel to~$C_{k}$.
Since none of those flips involves $C_k$, the same flips can be performed in $\AA$. 
Before those flips, at most $2\binom{n-k-1}{2}$ crossings 
are located between $C_k$ and $C_{k+1}$.
As each flip reduces this number by one, the expansion process for $C_{k+1}$ requires  
at most $2\binom{n-k-1}{2} < n^2$ triangle flips.
After performing these flips, $C_{k+1}$ is parallel to the entire bundle $\CC_{[k]}$.
We set $k := k+1$, $\CC_{[k+1]}:=\CC_{[k]} \cup \{C_{k+1}\}$, and proceed with Step~1.

\paragraph{Analysis.}

In Step~1, we reset the procedure at most once. 
Hence, each of Steps 2 to 5 is performed at most twice for each pseudocircle.

Step~2 (well-distribution) is performed for each of 
the $n$ pseudocircles at most once and each time $O(n^2)$ flips are
sufficient. In total, this gives $O(n^3)$ flips.

Step~3 (expand bundle) is a bit more involved.
When a bundle has reached $p$ we add the bundle to $\DD$ and restart with a new $C_1$, i.e., a new seed for a bundle. 
Suppose that bundles of sizes $n_1, n_2, \ldots n_b$ are added to~$\DD$. Then $n_1+n_2 + \ldots +n_b \le n$ holds. 
During processing a bundle $C_1,\ldots,C_{n_i}$, we perform
$(x_1+\ldots+x_{n_i})$ bundle flips.
Whenever we perform a bundle flip, this reduces the number of crossings which lie outside of the pseudocircles 
$C_1,\ldots,C_{n_i}$. 
Hence we have $x_1+\ldots+x_{n_i} \le 2 \binom{n}{2}$ and 
the number of triangle flips in Step~3 for a bundle of size~$n_i$ is
bounded by 
$(x_1+\ldots+x_{n_i}) \cdot 2  n_i \le 2  n^2 \cdot n_i$.
The number of triangle flips in Step~3 for the expansion of all $b$ bundles
can therefore be bounded by $2n^2 \cdot (n_1+\ldots+n_b) 
\le 2n^3.$

Step~4 (add to bundle) is again easy. 
For each of the $n$ pseudocircles, 
the overall expansion can be performed 
with at most $n^2$ flips. 
For all pseudocircles, this gives $O(n^3)$ flips in Step~3.

Summing up over all steps and multiplying the result by two for the possible case of one reset in Step~1,
this completes the proof for the cubic upper bound.

\paragraph{Lower bound.} 
We again consider the two canonical 
cylindrical arrangements $\AA^+_n$ and~$\AA^-_n$
from Section~\ref{sec:diam_cipa}
and show that their flip-distance in the triangle flip graph of intersecting arrangements is also $2 \binom{n}{3}$. 
Note that 
the triangle flip graph of cylindrical intersecting arrangements
is an induced subgraph of 
the triangle flip graph of intersecting arrangements.
Recall that each triple of pseudocircles in $\AA^-_n$ is a NonKrupp$(2)$ while each 
triple of pseudocircles in $\AA^+_n$ is a NonKrupp$(4)$.
In the triangle flip graph of arrangements of three pairwise intersecting pseudocircles,
the flip distance between
NonKrupp$(2)$ and NonKrupp$(4)$ is~2. 
Hence, at least $2 \binom{n}{3} = \Theta(n^3)$ triangle flips are needed to get from
$\AA^-_n$ to $\AA^+_n$. 
This completes the proof of \Cref{prop:diam_ipa}.
\end{proof}

%%%%%%%%%%%%%%%%%%%%%%%%%%%%%%%%%%%%%%%%%%%%%%%%%%
%%%%%%%%%%%%%%%%%%%%%%%%%%%%%%%%%%%%%%%%%%%%%%%%%%
\section{Characterization of Cylindrical Arrangements}
\label{sec:cylindrical}
%%%%%%%%%%%%%%%%%%%%%%%%%%%%%%%%%%%%%%%%%%%%%%%%%%
%%%%%%%%%%%%%%%%%%%%%%%%%%%%%%%%%%%%%%%%%%%%%%%%%%

In this section we study cylindrical arrangements in more detail and provide new cha\-rac\-teri\-za\-tions for them. 
The following technical lemma, a result of independent interest, will be convenient.
We remark that a similar result was also known to Snoeyink and Hershberger (cf.~\cite[Section~4.2]{SnoeyinkHershberger1991}). 
Let $\AA$ be an arrangement of pseudocircles, let $C$ be a pseudocircle
of~$\AA$, and let $x$ be a point on $C$ which is not a vertex. 
A flip of $C$ is 
\defi{$x$-restricted} if $x$ is not on the boundary of the flipped triangle or digon. 
Note that a digon-create flip can always be performed in an $x$-restricted manner.

\begin{lemma}
\label{cor:snoeyink_hershberger}
Let $\AA$ be an arrangement of pseudocircles, let $C $ be a pseudocircle in~$\AA$, and let $x$ be a point on~$C$ which is not a vertex. 
Then $C$ admits a sequence of 
$x$-restricted expanding flips such that at the end of the sequence~$C$ is incident to the unbounded cell of the arrangement.
\end{lemma}

\begin{proof}
If $x$ belongs to the unbounded cell of $\AA'=\AA \arrminus C$ we have nothing to show, hence, we assume that this is not the case. 
Consider a maximal sequence of $x$-restricted expanding flips applied to $C$ and assume that $C$ is not incident to the unbounded cell. In particular it cannot be the case that $\AA \arrminus C$ is fully contained in the interior of~$C$ and hence, using \Cref{lem:snoeyink_hershberger}, there must be an expanding flip for $C$, which by assumption is not $x$-restricted. Hence, it is a triangle or a digon-collapse flip that involves the edge of $C$ containing $x$. Let us denote this triangle or digon by $z_x$.

Now we transform $\AA$ via a stereographic mapping that makes $z_x$ the unbounded cell. Furthermore, let $z_o$ be the cell which was the unbounded cell before.  
Since $z_x \neq z_o$ and both remain in $\exterior{C}$ we again refer to \Cref{lem:snoeyink_hershberger} to conclude that $C$ 
admits an expanding flip. This must be a flip which was not possible before, hence, it is the flip of the cell $z_o$. This reveals that $C$ had been incident to the unbounded cell; a contradiction.
\end{proof}

We need a few more notions: The \defi{distance} between two cells $z$, $z'$ in an arrangement of pseudocircles is 
the minimum number of pseudocircles that a curve which starts in the interior of~$z$ 
and ends in the interior of $z'$ must cross. The \defi{eccentricity} of a cell $z$ 
is the maximum distance between $z$ and any other cell~$z'$. 

For the following characterization of cylindrical arrangements, we consider all 
pseudocircles to be oriented in counterclockwise direction, 
which induces an orientation on the edges of the arrangement. 
Consider the four intersecting 
arrangements of three pseudocircles (\Cref{fig:fourTypes}) with this orientation 
and note that only NonKrupp$(3)$ has a cell with a clockwise boundary.

\begin{restatable}{proposition}{propCylindricalArrangements}
\label{prop:cylindrical_arrangements}
    Let $\AA$ be an arrangement of $n$ pseudocircles with pairwise overlapping interiors.  
    Then, the following five statements are equivalent: 
    \begin{enumerate}[(1)]
	\itemsep0ex 
        \item $\AA$ is cylindrical.
	\item $\AA$ does not contain a NonKrupp$(3)$ as a subarrangement. 
        \item There is no clockwise oriented cycle in~$\AA$.
        \item There is no clockwise oriented cell in~$\AA$.
        \item The unbounded cell has eccentricity~$n$.
    \end{enumerate}
\end{restatable}

\begin{proof}
	$(1) \implies (4)$: Let $\AA$ be cylindrical. Then $\AA$ can be drawn as a pseudoparabola arrangement (\Cref{prop:cylinder}). Since all pseudoparabolas are $x$-monotone, all edges in the arrangement are oriented in a common direction. Hence, there cannot be a clockwise cycle, and in particular there cannot be a clockwise cell.
    
	$(4) \implies (3)$: Suppose that there is a clockwise oriented cycle. Let $z$ be a clockwise oriented cycle that is minimal with respect to the number of cells it contains and assume that $z$ is not a cell. Then there exists a pseudocircle $C$ that intersects $z$. Using an arc of $C$, the interior of $z$ can be split into two subregions, one of which must again be clockwise; a contradiction to the minimality of $z$.

	$(3) \implies (2)$: Every NonKrupp$(3)$ forms a clockwise oriented cycle by definition.

	$(2) \implies (1)$: For this implication we make use of the topological version of Helly's theorem 
 (cf.~\cite{ECKHOFF1993}). First, the intersection of the interiors of two pseudocircles in $\AA$ is by assumption non-empty and by definition simply connected. We claim that the intersection of the interiors of every three pseudocircles $C_1, C_2, C_3$ is non-empty. Let us first consider the case that two of the involved pseudocircles are non-intersecting, i.e., one is contained inside the other, say $\interior{C_1} \subset \interior{C_2}$. Then the intersection of $\interior{C_1}$ and $\interior{C_3}$ is also contained in $\interior{C_2}$. 
    It remains to consider the four intersecting subarrangements (see \Cref{fig:fourTypes}).  
    Since NonKrupp$(3)$ is excluded, the interiors of the three pseudocircles have a 
    common intersection.
    From the topological version of Helly's theorem, it then follows that the intersection of the interiors of all pseudocircles in $\AA$ is non-empty. 

	$(1) \implies (5)$: Again by \Cref{prop:cylinder}, we know that 
    $\AA$ can be represented 
    as an arrangement of pseudoparabolas. Due to the monotonicity of 
    the pseudoparabolas, in order to reach the bottom cell from the top 
    cell, one needs to traverse each of the $n$ pseudocircles, while the other cells require fewer traversals. Hence, the eccentricity
    of the unbounded cell equals $n$.

	$(5) \implies (1)$: 
    Let $\AA$ be an arrangement of $n$ pseudocircles which is not cylindrical. 
    Then we show that the eccentricity of the unbounded cell is 
    strictly less than $n$. Let $z$ be a bounded cell.  Since $\AA$ is not
    cylindrical, there is a pseudocircle $C$ in $\AA$ which has $z$ in  
    $\exterior{C}$. Our goal is to construct a new pseudocircle $C'$ from $C$, 
    which is incident to the unbounded cell, traverses $z$ and does not intersect 
    $C$. Such a $C'$ has at most $2(n-1)$ crossings, and therefore, along one of 
    the two arcs of $C'$ between $z$ and the unbounded cell, $C'$ has at 
    most $n-1$ crossings. This arc shows that the distance of $z$ and the unbounded cell is at most $n-1$. 
    
    It remains to construct $C'$. Begin with a slightly blown up parallel
    copy $C'$ of the pseudocircle~$C$, i.e., the ring formed by $C$ and $C'$ has no crossing in the interior and $C'$ as the outer boundary.     
    Expand $C'$ by performing flips into $\exterior{C'}$, this can be done by \Cref{lem:snoeyink_hershberger}. Eventually $C'$ will reach~$z$. Pick a point $x$ on the intersection of $C$ and $z$. Now continue with $x$-restricted 
    expanding flips. \Cref{cor:snoeyink_hershberger} implies that eventually $C'$ will become incident to the unbounded cell. Now $C'$ still traverses $z$ and is incident to the unbounded cell, and since $C$ was in $\interior{C'}$ initially and all the flips were expanding we still have $C \subset \interior{C'}$.    
\end{proof}

\section{Conclusion}\label{sec:conclusion}

In this work we studied arrangements of pseudolines and pseudocircles, in particular we 
contributed  new results regarding the connectivity of the corresponding 
triangle flip graphs. 

In the context of pseudoline arrangements, we studied the triangle flip graph $\mathbf{F}_n$. This graph was also studied by Gutierres,  Mamede, and Santos \cite{GuiterresMamedeSantos} as the graph of \textit{commutation classes} of \textit{reduced words} of the reverse of the identity permutation.
We proved a tight bound for the connectivity: $\mathbf{F}_n$ is  $(n-2)$-connected (for $n \le 7$ and~$n \ge 24$).
This proves \Cref{conjecture:signotopes} about the $(n-r+1)$-connectivity 
of the flip graph of $r$-signotopes on $n$ elements in the case $r=3$. 
For rank $r \ge 4$, \Cref{conjecture:signotopes} remains~open.

Another generalization of arrangements of pseudolines 
into higher dimensions are oriented matroids:
rank~$r$ oriented matroids can be represented as pseudohyperplane arrangements in the $(r-1)$-dimensional projective space. Mutations between oriented matroids are  flips of simplicial cells in the corresponding arrangements.
While the flip graph of realizable oriented matroids (arrangements of proper hyperplanes) 
and the flip graph of higher dimensional signotopes are connected \cite{Shannon1979,FelsnerWeil2001} 
(both are subgraphs of the flip graph of oriented matroids),
it remains a central problem in the theory of oriented matroids 
to verify whether every simple oriented matroid has a mutation 
(i.e., whether the flip graph contains isolated vertices);
\cite[Conjecture~7.3.10]{BjoenerLVWSZ1993}. 

Our proof of \cref{thm:connectivity} makes use of the fact that $\mathbf{F}_n$ contains a subgraph $\mathbf{F}_\Lambda$ that is the skeleton of a polytope. 
We recall a question from \cite{FeZi01}, which is still open:

\begin{question}
    Is the graph $\mathbf{F}_n$ polytopal for some $n\geq 6$?
\end{question}

It is also worth asking for the connectivity of the triangle flip graphs of arrangements in ambient spaces
other than the Euclidean plane, especially the following:
\begin{question}
    What is the connectivity of the 
    triangle flip graph of arrangements of $n$ pseudolines in the projective plane?
\end{question}

The number of simple arrangements of $n$ pseudolines and $n$ pseudocircles
(i.e., the number of vertices in the respective flip graphs) 
are both known to be of order $2^{\Theta(n^2)}$, see \cite{dumitrescuMandal2020,felsnerValtr2011,FelsnerScheucher2019}.
However,
determining better/precise bounds on the quadratic term in the exponent remains
a long-standing problem; see also the sequences \href{https://oeis.org/A006245}{A6245} and \href{https://oeis.org/A296406}{A296406} 
in the OEIS~\cite{OEIS}.

Another long-standing question is whether a natural triangle flip Markov chain on arrangements of pseudolines is rapidly mixing. The Markov chain $(X_t)$ is the lazy random walk in $\mathbf{F}_n$. Starting with an arbitrary arrangement $X_0$, 
the chain moves from $X_t$ to $X_{t+1}$ by selecting a triangle of $X_t$  uniformly at random and then flip it with probability $\frac{1}{2}$. (See \cite{levinPeresWilmer17} for an introduction to Markov chains and mixing~times).

\begin{question}
What can be said about the mixing time of the triangle flip Markov chain?
\end{question}

If the triangle flip Markov chain is rapidly mixing, then this could 
yield  a fully polynomial randomized approximation scheme (FPRAS) for counting simple arrangements of pseudolines. 
As the mixing time depends on the expansion of the underlying flip graphs of the Markov chain, our \Cref{thm:connectivity} may be a first step towards a proof that the  triangle flip Markov chain on arrangements is rapidly mixing.
A better understanding of short 
paths in $\mathbf{F}_n$ could further help in constructing a family of canonical paths, a common technique for bounding the mixing time of a Markov chain (see, for example, \cite[chapter~5]{jerrum03}). This motivates the following question which is of interest also from the reconfiguration point of view:

\begin{question}
What is the computational complexity of computing a shortest path between any pair of vertices in $\mathbf{F}_n$?
\end{question}

In the context of pseudocircle arrangements, we showed that the triangle flip graph of intersecting arrangements of~$n$ pseudocircles is connected for every $n \in \mathbb{N}$. 

The minimum degree in the triangle flip graph of intersecting arrangements of pseudocircles is known to be between $\frac{2n}{3}$ and $n-1$, see \cite{FelsnerScheucher2020}. Moreover, 
the minimum degree of the triangle flip graph of digon-free intersecting arrangements equals~$\lceil \frac{4n}{3} \rceil$, see \cite{FelsnerRS2022}. 
These results yield bounds for the connectivity of the respective flip graphs.

For \emph{digon-free} arrangements the analog of 
\Cref{question:pwi} was stated as a conjecture in 
\cite{FelsnerScheucher2019}:

\begin{question}
Is it true that, for every $n \in \mathbb{N}$, the triangle flip graph of digon-free intersecting arrangements of $n$ pseudocircles is connected? 
\end{question}

Felsner and Scheucher \cite{FelsnerScheucher2019} showed that the corresponding result for digon-free intersecting arrangements of circles is true. 
The idea is as follows: a given arrangement of circles is projected to a unit-sphere $\mathbf{S}$
and the circles are extended to planes (each circle is the intersection of a plane with~$\mathbf{S}$).
The planes are simultaneously moved towards the origin such that all circles eventually become great-circles.
The connectivity then follows from the flip connectivity of great-circle arrangements. 
When moving the planes towards the origin, it is crucial that each Krupp remains a Krupp and each NonKrupp is eventually transformed into a Krupp;  therefore no new digons~appear.

\begin{figure}[htb]
\centering
\includegraphics[page=1,scale=0.93]{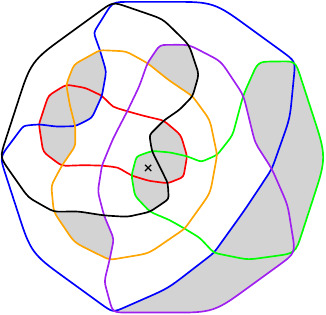}
\caption{The cylindrical arrangement $\AAsixB$. 
Each triangle (shaded gray) is formed by a Krupp. The center is marked by \textbf{$\times$}. 
}
\label{fig:n6_nonr_number2_noncyl}
\end{figure}

We remark that a corresponding approach to transform a digon-free arrangement of pairwise intersecting pseudocircles into an arrangement of great-pseudocircles by iteratively flipping NonKrupp triangles to increase the number of triples which form a Krupp does not seem to work. The 
arrangement~$\AAsixB$ shown in \Cref{fig:n6_nonr_number2_noncyl} is digon-free, contains a NonKrupp subarrangement (e.g., formed by the red, blue and black pseudocircles), but every triple of pseudocircles, which bound a triangle, form a Krupp. 
Note that $\AAsixB$ is 
non-circularizable~\cite[Section 6.2]{FelsnerScheucher2019}. 
By the connectivity result for digon-free intersecting circle arrangements an analogous situation cannot occur in circularizable arrangements.

With our \Cref{thm:main_cipa} we have shown that the triangle flip graph of intersecting cylindrical arrangements of $n$~pseudocircles is connected. The question whether the triangle flip graph of \emph{digon-free} intersecting cylindrical arrangements is connected remains open. 

The progress in the study of triangle flip graphs made in this paper
may foster the future study of questions like Hamiltonicity, degree of connectivity, expansion properties, and others for classes of triangle flip graphs.

Our focus in this paper was on arrangements of pseudolines and pseudocircles. Other authors have also studied other arrangements of 
curves with restricted intersection patterns, see e.g.\ \cite{AK23,Chan08,BHP08}. It seems, however, that flip graphs
have not been studied for families of curves beyond pseudolines and pseudocircles. Indeed one of the main tools in our studies,
the sweeping lemma is not available in general.
Snoeyink and Hershberger \cite{SnoeyinkHershberger1991} describe arrangements of curves with at most 3 and 4 pairwise intersections, respectively, which cannot be swept.

%%
%% Bibliography
%%
{
\small
\advance\bibitemsep-0pt
\let\sc=\scshape
\bibliographystyle{my-siam}
\bibliography{bibliography}
}

\end{document}